\documentclass[reqno,12pt,letterpaper]{amsart}
\usepackage{amsmath,amssymb,amsthm,graphicx,mathrsfs,url}
\usepackage[usenames,dvipsnames]{color}
\usepackage[colorlinks=true,linkcolor=Red,citecolor=Green]{hyperref}
\usepackage{floatrow}
\usepackage[usenames,dvipsnames]{xcolor}
\usepackage{tikz}
\usepackage{slashed}
\usepackage{graphicx}
\usepackage[colorlinks=true]{hyperref}
\usepackage{color}
\usepackage{amsmath,amsfonts}
\usepackage{calrsfs}
\usepackage{amsmath,environ}
\usepackage{floatrow}
\usepackage{tikz-cd}
\usepackage{bm}
\usepackage{autonum}

\DeclareMathAlphabet{\pazocal}{OMS}{zplm}{m}{n}

\numberwithin{equation}{section}
\usepackage{amsmath, amsthm}
    \usepackage{dsfont}
    \usepackage{fancybox}
    \usepackage{amssymb}

\newcommand{\DDi}{\slashed{\DD}}

\newcommand{\Di}{\slashed{D}}

\newcommand{\hf}{{\widehat{f}}}

\newcommand{\R}{\mathbb{R}}
\newcommand{\N}{\mathbb{N}}
\newcommand{\p}{{\partial}}

\newcommand{\w}{\omega}
\newcommand{\dd}[2]{\dfrac{\partial #1}{\partial #2}}
\newcommand{\matrice}[1]{\left[ \begin{matrix}
#1
\end{matrix} \right]}
\newcommand{\ove}[1]{\overline{#1}}

\newcommand{\var}{{\vartheta}}

\newcommand{\bsu}{\boldsymbol{\sigma_1}}
\newcommand{\bss}{\boldsymbol{\sigma_\star}}
\newcommand{\bsd}{\boldsymbol{\sigma_2}}

\newcommand{\bst}{\boldsymbol{\sigma_3}}
\newcommand{\bsi}{\boldsymbol{\sigma_i}}
\newcommand{\bsj}{\boldsymbol{\sigma_j}}
\newcommand{\PP}{{\mathscr{P}}}
\newcommand{\DD}{{\pazocal{D}}}

\newcommand{\loc}{{\text{loc}}}
\newcommand{\epsi}{\varepsilon}
\newcommand{\trace}{{\operatorname{Tr}}} 
\newcommand{\dist}{{\operatorname{dist}}} 
\newcommand{\oA}{{\operatorname{(H1)}}} 
\newcommand{\oB}{{\operatorname{(H2)}}} 
 
\newcommand{\Tr}{{\operatorname{Tr}}}

\newcommand{\te}{\theta}

\newcommand{\lr}[1]{\langle #1 \rangle}
\newcommand{\blr}[1]{\left\langle #1 \right\rangle}

\newcommand{\tvartheta}{{\widetilde{\var}}}
\newcommand{\tm}{{\widetilde{m}}}

\newcommand{\WW}{\pazocal{W}}

\newcommand{\Bb}{\mathbb{B}}

\newcommand{\tL}{\widetilde{L}}

\newcommand{\Z}{\mathbb{Z}}
\newcommand{\BBb}{\mathbb{B}}

\newcommand{\sgn}{{\operatorname{sgn}}}

\newcommand{\Pp}{\mathbb{P}}

\newcommand{\Ss}{\mathbb{S}}
\newcommand{\Id}{{\operatorname{Id}}}
\newcommand{\BB}{\pazocal{B}}
\newcommand{\VV}{{\pazocal{V}}}

\newcommand{\EE}{\pazocal{E}}

\newcommand{\MM}{\pazocal{M}}

\newcommand{\UU}{{\pazocal{U}}}

\newcommand{\LL}{{\pazocal{L}}}
\newcommand{\HH}{\pazocal{H}}

\newcommand{\vp}{{\varphi}}
\newcommand{\hVV}{{\widehat{\VV}}}
\newcommand{\hWW}{{\widehat{\WW}}}
\newcommand{\VVV}{{\mathscr{V}}}
\newcommand{\WWW}{{\mathscr{W}}}

\newcommand{\systeme}[1]{\left\{ \begin{matrix} #1 \end{matrix} \right.}

\newcommand{\End}{{\operatorname{End}}}

\newcommand{\C}{\mathbb{C}}

\newcommand{\oz}{{\overline{z}}}

\newcommand{\az}{\alpha}

\newcommand{\OO}{{\pazocal{O}}}

\newcommand{\Sf}{{\operatorname{Sf}}}
\newcommand{\CCC}{{\mathscr{C}}}
\renewcommand{\Re}{\operatorname{Re}}

\renewcommand{\Im}{\operatorname{Im}}

\newcommand{\bblr}[1]{\big\langle #1 \big\rangle}
\newcommand{\Mm}{\mathbb{M}}

\newcommand{\ou}{\overline{u}}

\newcommand{\Tt}{{\mathbb{T}}}
\newcommand{\Ee}{{\mathbb{E}}}

\newcommand{\ev}{{\operatorname{e}}}
\newcommand{\od}{{\operatorname{o}}}

\newcommand{\tE}{{\widetilde{E}}}

\newcommand{\1}{\mathds{1}}

\newcommand{\de}{ \ \mathrel{\stackrel{\makebox[0pt]{\mbox{\normalfont\tiny def}}}{=}} \ }

\title{The bulk-edge correspondence for continuous dislocated systems}

\NewEnviron{equations}{%
\begin{equation}\begin{gathered}
  \BODY
\end{gathered}\end{equation}
}

\newtheorem{thm}{Theorem}
\newtheorem{conj}{Conjecture}
\newtheorem{cor}{Corollary}
\newtheorem{lem}{Lemma}[section]

\newtheorem{definition}[lem]{Definition}
\newtheorem{theorem}[thm]{Theorem}
\theoremstyle{definition}
\newtheorem{rmk}{Remark}[section]

\author{Alexis Drouot}

\begin{document}
\maketitle

\begin{abstract} We study topological aspects of defect modes for a family of operators $\{\PP(t)\}_{t \in [0,2\pi]}$ on $L^2(\R)$. $\PP(t)$ is a periodic Schr\"odinger operator $P_0$ perturbed by a \textit{dislocated} potential. This potential is periodic on the left and on the right, but acquires a phase defect $t$ from $-\infty$ relative to $+\infty$. When $t=\pi$ and the dislocation is \textit{small and adiabatic,} Fefferman, Lee-Thorp and Weinstein \cite{FLTW1,FLTW2} showed that Dirac points of $P_0$ (degeneracies in the band spectrum of $P_0$) bifurcate to defect modes of $\PP(\pi)$.

We show that these modes are \textit{topologically protected} at the level of the family $\{\PP(t)\}_{t \in [0,2\pi]}$. This means that local perturbations cannot remove these states for all values of $t$, \textit{even outside the small adiabatic regime} of \cite{FLTW1,FLTW2}. We define two topological quantities: an edge index (the signed number of eigenvalues crossing an energy gap as $t$ runs from $0$ to $2\pi$) and a bulk index (the Chern number of a Bloch eigenbundle for the periodic operator near $+\infty$). We prove that these indexes are equal to the same \textit{odd} winding number. This shows the bulk-edge correspondence for the family $\{\PP(t)\}_{t \in [0,2\pi]}$. We express this winding number in terms of the asymptotic shape of the dislocation and of the Dirac point Bloch modes. We illustrate the topological depth of our model via the computation of the bulk/edge index on a few examples. 
\end{abstract}

\section{Introduction}

The study of waves in crystalline materials has a long history. Central questions include energy transport, localization, inverse problems, scattering by impurities, etc. An important theme concerns whether two different crystals  can be deformed into one another while preserving their electronic properties. This amounts to topologically classifying periodic materials. This classification was first believed to depend solely on the band spectrum of the crystals. But the scientific community soon realized that the Bloch modes encoded essential information. This led to theoretical explanations of the quantum Hall effect \cite{Hal,TKN,Hat}; discoveries of topological insulators \cite{KM1,KM2,FKM,MB,Roy}; and topological interpretations of certain oceanic and atmospheric waves \cite{DMV,PDV}. 

Highly symmetric crystals are central to the field of topological insulators. They exhibit degeneracies in their band spectrum. Symmetry-breaking perturbations may remove these degeneracies and open essential spectral gaps. When this happens, metals become insulators. The existence of energy gaps allows to define low-energy eigenbundles. Their topology depend on their Chern classes. For two-dimensional materials, there is a single Chern class, represented by an integer: the bulk index. 

In favorable conditions, certain half-crystals are insulating in the bulk but conducting along their boundary. In other words, they support time-harmonic waves propagating along rather than across the edge. According to numerical and physical experiments, this property is strikingly robust: it persists under large deformations of the interface. This inspired investigations in electronic physics, photonics, acoustics and mechanics; see e.g. \cite{KMY,YVW,WCJ,SG11,RZP,I15,M17,O18}.
 From the theoretical point of view, this motivated the definition of the edge index: the difference between the number of edge states propagating up and down. 

The bulk-edge correspondence is a formal principle that asserts that the bulk and edge indexes coincide. Such identities are ubiquitous in mathematics and physics: they relate spectral quantities (here, the edge index) to topological quantities (the bulk index). One of the most famous example is the Atiyah--Patodi--Singer index theorem \cite{APS1,APS2,APS}. The bulk-edge correspondence has been proved in a growing number of discrete settings \cite{EGS,ASV,GP,PS,Ba,Sha,Br,GS,GT,ST}; see \cite{FC,AOP} for a good introduction. There is fewer work on continuous systems: see \cite{KS,KS2,Taa} for the quantum Hall effect; \cite{FSF,Ba,Ba2} for an analysis on Dirac operators; and \cite{BR} for a K-theory approach. Here, we prove the bulk-edge correspondence in a new continuous setting.

The system models the dislocation of a one-dimensional crystal; equivalently, the imperfect junction of two similar crystals at a phase defect $t \in [0,2\pi]$. For small and adiabatic dislocations, Fefferman, Lee-Thorp and Weinstein \cite{FLTW1,FLTW2} constructed defect states via a rigorous multiscale procedure. These defect states bifurcate from Dirac points (conical degeneracies in the band spectrum) of the periodic background. We recently showed that the multiscale  approach of \cite{FLTW1,FLTW2} produces all defect states \cite{DFW}. 

Using \cite{DFW}, we derive a formula for the edge index. It depends only on the bulk: it is a winding number expressed in terms of the asymptotic shape of the dislocation and of the Dirac point Bloch modes of the periodic background. It shows that Dirac points play a key role in the production of defect states.  Because of topological invariance, our formula gives the edge index \textit{regardless of small or adiabatic assumptions.}  

We then use the classical expression of the bulk index as the integral of the Berry curvature. The same winding number emerges via the rigorous reduction to a tight-binding model. This proves the bulk-edge correspondence for our model. It extends and generalizes parts of results of Korotyaev \cite{Ko1,Ko2}, Dohnal--Plum--Reichel \cite{DPR}  and Hempel--Kohlmann \cite{HK}.  We compute the bulk/edge index on a few examples. This demonstrates that the bulk/edge index of our system can take any odd value.

Our model shares many features with honeycomb lattices glued along an edge. Fefferman, Lee-Thorp, Weinstein and Zhu successfully exploited this connection to construct edge states in the small adiabatic regime \cite{FLTW3,FLTW4,LWZ}. Their work demonstrates mathematically some of the predictions of Haldane and Raghu \cite{HR1,RH}. The work of Hempel and Kohlmann \cite{HK} inspired studies for two-dimensional dislocations \cite{HK2,HK4,HK3,KM}. Our techniques will likewise apply to more advanced models.

\subsection{Description of the model and motivation}\label{sec:1.1}
We study a family of Schr\"odinger operators $\{\PP(t)\}_{t \in [0,2\pi]}$ on $L^2(\R)$ with potentials periodic on the far left and the far right, but acquiring a phase defect $t$ when going from $+\infty$ to $-\infty$. Specifically:
\begin{equation}\label{eq:1l}
\PP(t) \de -\dd{^2}{x^2}  + V + \chi_-   W + \chi_+  W_t + F, \ \ \ \ \ \ \text{where:}
\end{equation}
\begin{itemize}
\item $V$ and $W$ are smooth, real-valued, one-periodic and satisfy
\begin{equation}\label{eq:9v}
V(\cdot+1/2) = V(x); \ \ \ \ W(x+1/2) = -W(x); \ \ \ \ W_t(x) \de W\big( x + t/(2\pi) \big).
\end{equation}
\item $\chi_\pm$ are smooth and real-valued and satisfy
\begin{equation}
\chi_+ + \chi_- = 1, \ \ \ \ \chi_+(x) = \systeme{1 & \text{ for } x \gg 1 \ \ \ \\ 0 & \text{for } x \ll -1}, \ \ \ \ \chi_-(x) = \systeme{1 & \text{for } x \ll -1 \\ 0 & \text{ for } x \gg 1 \ \ \ }.
\end{equation}
\item $F$ is a selfadjoint operator on $L^2(\R)$, compact from $H^2(\R)$ to $L^2(\R)$.
\end{itemize}
The operator $\PP(t)$ is the dislocation of a periodic Schr\"odinger operator $P_0 \de D_x^2 + V$ by a potential that looks like $W$ for $x \ll -1$ and $W_t$ for $x \gg 1$. See Figure \ref{fig:3}.

The work of Fefferman, Lee-Thorp and Weinstein \cite{FLTW1,FLTW2} motivates our analysis. These papers study defect modes for 
\begin{equation}\label{eq:3t}
\PP_{\delta,\star} \de D_x^2 + V - \delta \cdot \kappa(\delta \ \cdot) \cdot W, \ \ \ \ D_x \de \dfrac{1}{i} \dd{}{x}, \ \ \ \ \kappa \de \chi_+-\chi_-.
\end{equation} 
This corresponds to \eqref{eq:1l} when $F \equiv 0$, $t=\pi$, $\chi_+$, $\chi_-$ vary at scale $\delta^{-1} \gg 1$ and $W = O(\delta)$. The operator \eqref{eq:3t} is a small non-local perturbation of $P_0 = D_x^2 + V$. As a periodic operator, $P_0$ acts on the space of quasi-periodic functions
\begin{equation}\label{eq:8a}
L^2_\xi \de \left\{ u\in L^2_\loc : u(x+1) = e^{i\xi} u(x) \right\}.
\end{equation}
These are Hilbert spaces when provided with the bilinear form $\blr{f,g}  \de \int_0^1 \ove{f(x)} g(x) dx$. Because of \eqref{eq:9v}, the $L^2_\pi$-eigenvalues of $P_0$ are two-fold degenerate. They correspond to Dirac points: conical intersections of the $L^2_\xi$-eigenvalues of $P_0$. If $(\pi,E_\star)$ is a Dirac point, there exist normalized solutions $\phi_\pm^\star \in L^2_\pi$ of the eigenvalue problems
\begin{equation}
\left(D_x^2+V-E_\star\right) \phi_\pm^\star = 0, \ \ \ \ \phi_\pm^\star(x+1/2) = \pm i \phi_\pm(x).
\end{equation}
These solutions are uniquely defined modulo multiplicative factors in $\Ss^1$. We call $(\phi_+^\star,\phi_-^\star)$ a ``Dirac eigenbasis". We refer to \S\ref{sec:3.0} for details.

\begin{center}
\begin{figure}
{\caption{The periodic potentials $V$ (red) and $W$ (blue) are plotted in the top graph; on the right, $W$ acquires a phase defect $t$. The bottom graph represents the typical potential in the Schr\"odinger operator $\PP(t)$: it is periodic on the left and on the right, with different asymptotics. The gray transition zone takes an arbitrary form.}\label{fig:3}}
{\includegraphics{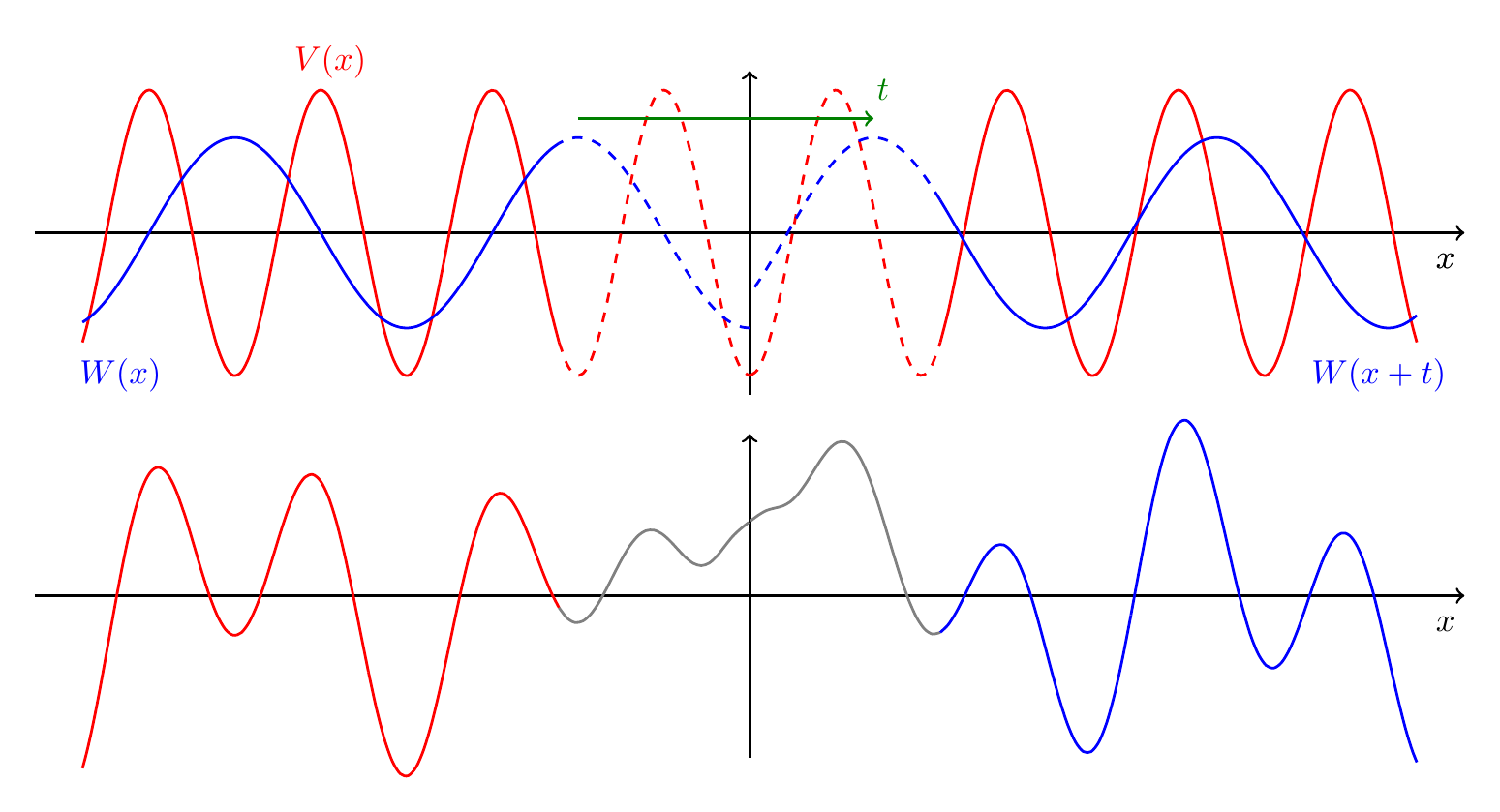}}
\end{figure}
\end{center}    
\vspace*{-.9cm}

Introduce the parameter
\begin{equation}
\var_\star \de \blr{\phi_-^\star, W_\pi \phi_+^\star}  = \int_0^1 \overline{\phi_+^\star(x)} W\big( x + 1/2 \big) \phi_-^\star(x) dx.
\end{equation}
When $\var_\star \neq 0$ and $\delta > 0$ is sufficiently small, the operator $\PP_{\delta,\star}$ has an essential spectral gap centered at the Dirac energy $E_\star$, of size $2 \var_F \delta + O(\delta^2)$, where $\var_F = |\var_\star|$. The work \cite{FLTW2} predicts that $\PP_{\delta,\star}$ admits at least one eigenvalue in each of these gaps, equal to $E_\star + O(\delta^2)$. The corresponding eigenvector is seeded by the zero mode of the Dirac operator
\begin{equation}
\DDi_\star \de \nu_\star \bst  D_x + \bss  \kappa, \ \ \ \ \nu_\star \de 2\blr{\phi_+^\star,D_x\phi_+^\star}, \ \ \ \ \bss \de \matrice{0 & \ove{\var_\star} \ \\ \var_\star & 0}, \ \ \bst \de \matrice{1 & 0 \\ 0 & -1}.
\end{equation}
The operator $\DDi_\star$ emerges via a multiscale analysis of \eqref{eq:3t}.
We improve this result in \cite{DFW}: we show that the eigenvalues of $\DDi_\star$ seed precisely \textit{all} the eigenvalues of $\PP_{\delta,\star}$. See \S\ref{sec:3.3} for precise statements and Figure \ref{fig:2} for a pictorial representation.

The zero mode of the Dirac operator $\DDi_\star$ persists against large compact perturbations of $\kappa$. A natural question is whether this robustness property transfers to the defect states constructed in \cite{FLTW2} -- in other words, whether these modes persist against large deformations.  The short answer is no: one could deform these modes away by adding a suitable finite-rank operator. 

A less naive answer consists in \textit{embedding} $\PP_{\delta,\star}$ in a family of dislocation problems of the form \eqref{eq:1l}. This family is $2\pi$-periodic with respect to $t$ and topological effects appear. We show that under suitable assumptions, the family $t \in [0,2\pi]\mapsto \PP(t)$ admits defect states independently of the perturbation $F$. This means that there are always phase shifts $t$ such that $\PP(t)$ admits eigenvalues. In particular, the corresponding modes persist against arbitrary bulk-preserving perturbations.

\subsection{Existence of defect states} Fix $n$ an odd integer. Because $V$ is $1/2$-periodic, the $n$-th and $n+1$-th dispersion curves of $D_x^2 + V$ join at a Dirac point $(\pi,E_\star)$, see the above discussion and \S\ref{sec:3.0}. In addition to \eqref{eq:9v}, we introduce two assumptions for the pair of potentials $(V,W)$:
\begin{enumerate}
\item[$\oA$] There exists $E \in C^0\big( [0,1] \times \R/(2\pi \Z), \R \big)$ with
$E(s,t) = E_\star$ for $s$ small and 
\begin{equation}
 \ \ \ s \in (0,1], \ t \in \R \ \Rightarrow \ 
E(s,t) \text{ is not in the } L^2_\pi\text{-spectrum of } D_x^2 + V + s W_t.
\end{equation}
\item[$\oB$] For every $t \in \R$, $\var(t) \de \blr{\phi_-^\star, W_t \phi_+^\star}$ is non-zero.
\end{enumerate}
These conditions concern only the bulk structure of the problem; if they hold for $(V,W)$ they also hold for small perturbations of $(V,W)$. $\oA$ means that as $s,t$ vary, the $n$-th gap of $D_x^2 + V + s W_t$ remains open; $\oB$ is a non-degeneracy condition that already appeared in \cite{FLTW1,FLTW2,DFW}. In \S\ref{sec:8}, we provide examples of pairs $(V,W)$ where both $\oA$ and $\oB$ hold.  $\oB$ allows to define the degree of $\var : \R \rightarrow \C^*$:
\begin{equation}\label{eq:1z}
m \de \dfrac{1}{2\pi i} \int_0^{2\pi} \dfrac{\var'(t)}{\var(t)} dt.
\end{equation}
Equivalently, $m$ is the winding number of $t \in \Tt^1 \mapsto \var(t) \in \C^*$ around $0$. Since $\var(t+\pi) = -\var(t)$, $m$ is odd (see \S\ref{sec:2.4}); in particular $m \neq 0$. A simple version of our main results predicts existence of defect states for $\PP(t)$.

\begin{cor}\label{cor:1} If $(V,W)$ satisfies \eqref{eq:9v}, $\oA$ and $\oB$ then there exist at least $|m|$ values of $t \in [0,2\pi]$ such that $E(1,t)$ is an eigenvalue of $\PP(t)$ -- counted with multiplicity.
\end{cor}

\subsection{Bulk-edge correspondence} We describe our main results. We start by defining indexes. The edge index is the spectral flow $\Sf_{\Tt^1}(\PP-E(1,\cdot))$, where $\Tt^1$ denotes the torus $\R/(2\pi\Z)$. Roughly speaking,  it is the spectral multiplicity that effectively crosses the gap \textit{downwards}\footnote{Different authors picked different conventions. For instance, in \cite{APS,P,GLM} the spectral flow is the \textit{upwards} net crossing of eigenvalues. We chose the ``downwards" convention because it has been previously used in the mathematical litterature about bulk-edge correspondence \cite{Br} and dislocation problems \cite{HK}, both of which have strong connections with the present work.} as $t$ runs from $0$ to $2\pi$; see \S\ref{sec:4.0} for a precise definition. Corollary \ref{cor:1} is a direct consequence~of:

\begin{center}
\begin{figure}
{\caption{This picture represents the spectrum of $\PP(t)$ as a function of $t$. The gray zone corresponds to the essential spectrum; eigenvalues are plotted in red; $t \mapsto E(1,t)$ is an arbitrary energy level in the gap. The spectral flow counts the signed number of downwards intersections between the red and blue curves. Here it equals $1$: the first crossing occurs downwards; the upper curve does not intersect $E(1,t)$; and the lower curve intersections cancel each other. A bulk-preserving perturbation must leave the spectral flow invariant: the first red curve cannot be deformed away.}\label{fig:1}}
{\includegraphics{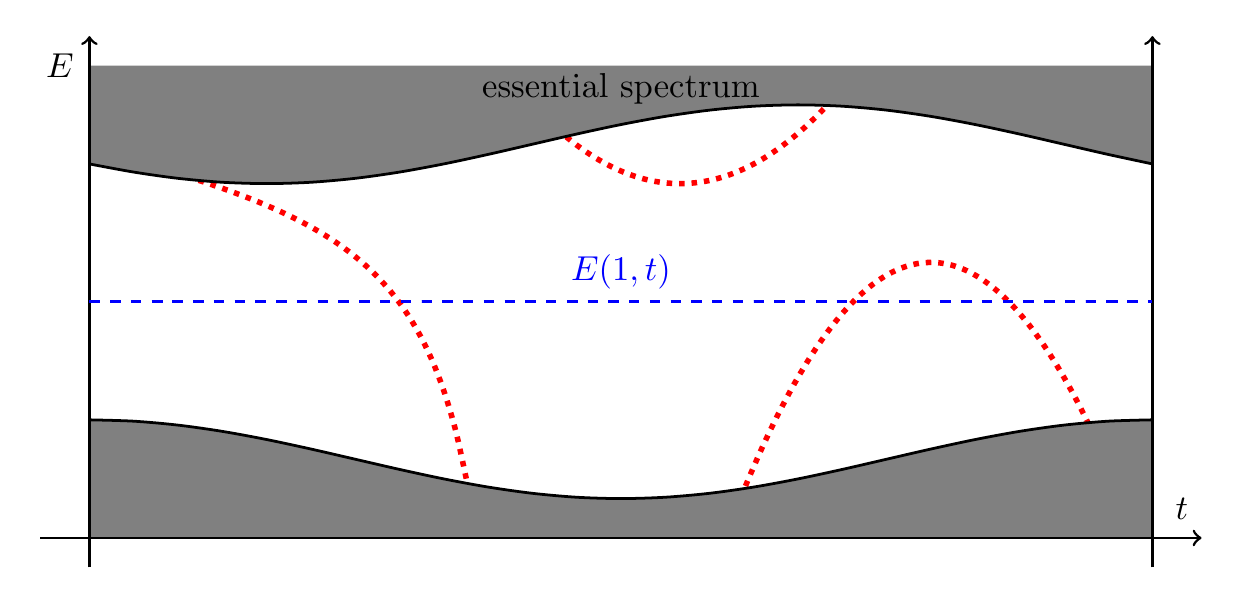}}\end{figure}
\end{center}
\vspace*{-.7cm}

\begin{theorem}\label{thm:1} If the pair $(V,W)$ satisfies \eqref{eq:9v}, $\oA$ and $\oB$ then
\begin{equation}
\Sf_{\Tt^1}\big(\PP-E(1,\cdot)\big) = (-1)^{\frac{n-1}{2}} \cdot m.
\end{equation}
\end{theorem}

We now define the bulk index. As a periodic operator, $D_x^2 + V + W_t$ acts on $L^2_\xi$ -- the space of quasi-periodic functions \eqref{eq:8a}. Because of (H1), the $n$-th gap of $D_x^2 + V + W_t$ on $L^2_\xi$ is open. We can then define a smooth rank-$n$ bundle $\EE$ over the two-torus $\Tt^2 = \R^2/(2\pi\Z)^2$: the fibers at $(\xi,t) \in \Tt^2$ is the direct sum of the first $n$ eigenspaces of $D_x^2 + V + W_t$ on $L^2_\xi$. The bulk index is the first Chern class $c_1(\EE)$ of $\EE$.

\begin{theorem}\label{thm:2} If the pair $(V,W)$ satisfies \eqref{eq:9v}, $\oA$ and $\oB$ then
\begin{equation}
c_1(\EE) = (-1)^{\frac{n-1}{2}}\cdot m.
\end{equation}
\end{theorem}

Theorem \ref{thm:1} and \ref{thm:2} give the bulk-edge correspondence: the bulk prescribes the signed number of defect states of the family $\{\PP(t)\}_{t \in [0,2\pi]}$. Furthermore, they express the bulk and edge indexes in terms of $W$, $n$ and $\phi_\pm^\star$ only:
\begin{equation}
\Sf_{\Tt^1}\big(\PP-E(1,\cdot)\big) = c_1(\EE) =  \dfrac{(-1)^{\frac{n-1}{2}}}{2\pi i} \int_0^{2\pi} \dfrac{\var'(t)}{\var(t)} dt.
\end{equation}
This formula indicates that the existence of defect states depends exclusively on the rough shapes of the perturbation $W$ and of the Dirac eigenbasis $(\phi_+^\star,\phi_-^\star)$ -- rather than on the details of $V$. This is another form of the topological manifestation identified here.

\subsection{Relation with \cite{FLTW1,FLTW2,DFW}} Theorem \ref{thm:1} is -- in nature -- a topological result rather than an asymptotic result: no smallness of parameters is assumed. It predicts the existence of defect states for a much wider class of operators than those considered in \cite{FLTW1,FLTW2,DFW}. However, it neither specifies their exact number, their shape, nor which values of $t$ guarantee their existence. This is in sharp contrast with \cite{FLTW1,FLTW2,DFW}. The first two papers describe defect states of $\PP_{\delta,\star}$ as two-scale functions oscillating like $(\phi_+^\star,\phi_-^\star)$, with slow decay prescribed by eigenvectors of $\DDi_\star$. The third paper refines this description with the precise counting of defect states; and full expansions of eigenpairs in powers of $\delta$. Both these results are asymptotic: they work for small $\delta$. The strength of Theorems \ref{thm:1} and \ref{thm:2} is their robustness. These results identify defect modes that persist against \textit{all} bulk-preserving perturbations, for $\delta$ of order $1$. Their proofs contain topological and analytical steps: homotopy arguments and reduction to Dirac operators. The present paper justifies the ``topologically protected" terminology used in \cite{FLTW2}. 

Our one-dimensional model shares many features with higher-dimensional analogs. This connection was succesfully exploited in \cite{FLTW3,FLTW4,LWZ} to construct edge states in continuous honeycomb structures. We will apply the present analysis to show existence of topological edge states in honeycomb lattices.

\subsection{Proof of Theorem \ref{thm:1}} The proof of Theorem \ref{thm:1} consists of three mains steps:
\begin{itemize}
\item In \S\ref{sec:2}, we deform continuously the family $t \mapsto \PP(t)-E(t,1)$ to
\begin{equation}\label{eq:5a}
\PP_\delta(t) \de D_x^2 + V + \delta \chi_-(\delta \ \cdot) W + \delta \chi_+(\delta \ \cdot) W_t  - E_\star,
\end{equation}
where $\delta > 0$ is arbitrarily small. This preserves the spectral flow associated to the $n$-th gap because of assumption $\oA$.
\item The family \eqref{eq:5a} is spectrally connected to effective Dirac operators
\begin{equations}
t \in \Tt^1 \mapsto \DDi(t) = \nu_\star\bst  D_x - \bss  \chi_- + \matrice{0 & \ove{\var(t)} \ \\ \var(t) & 0} \chi_+ \ : \ H^1(\R) \rightarrow  L^2(\R),\\
\nu_\star \de 2 \blr{\phi_+^\star,D_x\phi_+^\star}, \ \ \ \ \var(t) \de \blr{\phi_+^\star, W_t \phi_-^\star} \neq 0 \ \text{ by  } \ \oB. 
\end{equations}
The operator $\DDi(t)$ emerges from a multiscale analysis of $\PP_\delta(t)$ -- see \cite{FLTW1,FLTW2}. 
The key tool is the analysis of \cite{DFW} which shows that the defect states of $\PP_{\delta,\star}$ with energy near $E_\star$ are in bijection with the modes of the effective Dirac operator $\DDi_\star = \DDi(\pi)$. This prediction applies similarly to $\PP_\delta(t)$, $\DDi(t)$. This allows to count precisely the eigenvalues of $\PP_\delta(t)$ in terms of those of $\DDi(t)$. We derive the spectral flow equality
\begin{equation}
\Sf_{\Tt^1} \big(\PP_\delta-E_\star\big) = \Sf_{\Tt^1}\big(\DDi\big).
\end{equation}
\item The operator $\DDi(t)$ depends on $t$ only through the function $t \mapsto \var(t)$, which is homotopic to $t \mapsto -\var_\star e^{imt}$ ($m$ is the degree of $\var$). Hence, $\DDi$ deforms continuously to a simpler Dirac operator
\begin{equation}
\Di_m(t) \de \sgn(\nu_\star) \bst  D_x - \bss  \chi_- - \matrice{0 & \ove{\var_\star}e^{-imt} \\ \var_\star e^{imt} & 0} \chi_+ \ : \ H^1(\R) \rightarrow  L^2(\R)
\end{equation}
again without changing the spectral flow.
\item The spectral flow of $t \mapsto \Di_m(t)$ equals $m$ times the spectral flow of $\Di_1(t)$. We replace $\chi_+$ (respectively $\chi_-$) in $\Di_1(t)$ by $\1_{[0,\infty)}$ (resp. $\1_{(-\infty,0]}$). This preserves the spectral flow; and reduces $\Di_1(t)$ to an operator with constant coefficients. An explicit calculation is possible:
\begin{equation}\label{eq:5b}
\Sf_{\Tt^1}\left(\Di_1\right) = \sgn(\nu_\star).
\end{equation}
The conclusion comes from the identity $\sgn(\nu_\star) = (-1)^{\frac{n-1}{2}}$.
\end{itemize}

\subsection{Proof of Theorem \ref{thm:2}} In \S\ref{sec:4}, we define the bulk index. We adopt a geometric point of view. First we review the Berry connection and curvature of eigengbundles over the two-torus.  The bulk index is the (first) Chern class of the first $n$ eigenbundles $\EE$. We represent this topological invariant as the integral of the Berry curvature. Let $P(\xi,t) = D_x^2 + V +W_t$, acting on $L^2_\xi$. In \S\ref{sec:5}, we compute the bulk index of $(\xi,t) \mapsto P(\xi,t)$ associated to the $n$-th gap. We deform $P(\xi,t)$ to an operator $P_{\delta}(\xi,t)$ whose $n$-th gap has width of order $\delta \ll 1$. The $n$-th eigenvalue of $P_{\delta}(\pi,t)$ is almost degenerate. The physical intuition suggests that the Berry curvature is ``largest" near $(\pi,t)$. For small $\delta$ and $\xi$ near $\pi$, the operator $P_{\delta}(\xi,t)$ behaves like a $2 \times 2$ tight-binding model $\Mm_\delta(\xi,t)$. We then show that the bulk index equals the Chern number for the low-energy eigenbundle of $\Mm_\delta(\xi,t)$. A direct calculation yields then Theorem \ref{thm:2}:
\begin{equation}
c_1(\EE) = \dfrac{(-1)^{\frac{n-1}{2}}}{2\pi i} \int_0^{2\pi} \dfrac{\var'(t)}{\var(t)} dt.
\end{equation}

Although the proofs of Theorems \ref{thm:1} and \ref{thm:2} are independent, they share the same ground: the reduction to a Dirac operator. We perform this reduction by shrinking the amplitude of $W$ to zero. The assumptions $\oA$ and $\oB$ ensure that the essential spectral gap does not close. This effectively removes all terms that are not related to the Dirac point $(\pi,E_\star)$.

\subsection{Comparison with earlier work} Papers due to Korotyaev \cite{Ko1,Ko2}, Dohnal--Plum--Reichel \cite{DPR} and Hempel--Kohlmann \cite{HK} previously studied  the dislocation problem on the line, i.e. the eigenvalues of
\begin{equation}\label{eq:1d}
Q(t) = D_x^2 + \1_{(-\infty,0]} \WWW + \1_{[0,\infty)} \WWW\big(\cdot + t/(2\pi)\big),
\end{equation}
where $\WWW$ is a one-periodic potential.\footnote{\cite{Ko2} is more general: it considers cases where the potentials on the left and right are different.} The main difference of \eqref{eq:1d} with $\PP(t)$ is the absence of a periodic background potential. It has fundamental implications.

The essential spectrum of $Q(t)$ coincide with the spectrum of the periodic operator $Q(0) = D_x^2 + \WWW$. When the $n$-th gap is open, Korotyaev \cite{Ko1} and Dohnal--Plum--Reichel \cite{DPR} show that there exist $0 = t_0 < t_1 < \dots <t_n = 1$ such that for every $t \in (t_j,t_{j+1})$, the operator $Q(t)$ has a unique eigenvalue in the $n$-th gap. The effective flow of eigenvalues of $Q(t)$ in the $n$-th gap is $n$. Hempel--Kohlmann \cite{HK} rediscovered this formula via a clever trick, flexible enough to work in two-dimensional analogs \cite{HK2,HK3}. \cite{DPR} provides many interesting numerical examples. We do not know whether the results of \cite{Ko1,DPR} persist against localized perturbations; though the spectral flow interpretation remains. 

In contrast with Theorem \ref{thm:1}, \cite{Ko1} requires no hypothesis on $\WWW$ except -- naturally -- an open $n$-th gap for $D_x^2+\WWW$. Our paper requires more assumptions, in particular \eqref{eq:9v}; but offers a different perspective:
\begin{itemize}
\item[(a)] we  add the periodic potential $V$ to \eqref{eq:1d}; 
\item[(b)] we investigate the topological relation between bulk and edge  -- a question not addressed in \cite{Ko1,Ko2,DPR,HK}. 
\end{itemize}

Because of (a), our class of model is topologically richer than that of \cite{Ko1,DPR,HK}. For instance, in \S\ref{sec:8} we produce for each pair of odd numbers $(m,n) \neq (-n,n)$ two potentials $V$ and $W$ such that the spectral flow of $\{\PP(t)\}_{t \in [0,2\pi]}$ in the $n$-th gap equals $m$. When $m \neq n$, this shows that $\{\PP(t)\}_{t \in [0,2\pi]}$ is topologically different from the families studied in \cite{Ko1,DPR,HK}. This indicates that some assumptions -- probably in the spirit of $\oA$ -- are necessary for Theorem \ref{thm:1} to hold. It also demonstrates that Theorem \ref{thm:1} and \ref{thm:2}  provide a way to topologically distinguish pairs of potentials $(V,W)$:  the data of the bulk-edge indexes $m_n$ associated to the $n$-th gap. 

In additional to  $\Z$-translational invariance, the operator $P(\xi,t)$ seems to exhibit generically only one symmetry: the complex conjugation. Specifically,
\begin{equation}
f \in L^2_{\xi} \ \ \Rightarrow \ \ \overline{P(\xi,t) f} = P(-\xi,t) \overline{f}.
\end{equation}
This is not the time-reversal symmetry, which demands for $P(-\xi,-t)$ on the RHS. The Kitaev table \cite{Ki,RSFL} seems to indicate that our model is in the same topological class as the quantum Hall effect -- after spectral projection to a given gap. The quantum Hall effect is characterized by a $\Z$-index. This suggest that the data of Chern numbers $\{m_n : n \geq 1\}$ should be enough to classify $(V,W)$ topologically. 

Our analysis relates to the study of waves scattered by impurities in periodic backgrounds, see e.g. \cite{DS,DH,FK,Ol,Bor,BG1,PLAJ,Bo7,HW,Bo8,Ze,WW}. In comparison with these papers, we focus here on topological aspects. The relation between spectral and topological aspects of linear PDEs has a long and rich history, with ongoing developments. We would like to mention some recent advances (with no attempt to make an exhaustive list):
\begin{itemize}
\item For Anosov flow on $3$-manifolds, \cite{DGRS} recently proved the equality between the analytic torsion and the dynamical zeta function at $0$ (see \cite{SM} for the analytic case). 
\item On a related note,  \cite{DZ2,Had} investigated the connection between the Hodge--de Rham cohomology and $k$-forms invariant under the Anosov flow.
\item The celebrated Atiyah--Patodi--Singer index theorems have now proofs that do not rely on ellipticity \cite{BS,Br2}.
\item In the sub-Riemannian context, \cite{Bar} studied the small-time asymptotic of the heat kernel in terms of topological invariants.  
\end{itemize}

\subsection{Further perspectives} 

This work suggests other directions of research:
\begin{itemize}
\item Assume that $(V,W)$ and $(\VV,\WW)$ are two pairs of potentials satisfying \eqref{eq:9v}, $\oA$ and $\oB$ in the $n$-th gap; but that the winding numbers defined by \eqref{eq:1z} are different. Then we expect non-topological transitions when one goes from $(V,W)$ to $(\VV,\WW)$. These relate to Chern number transfers at linear band crossings. See \cite{Fa,FZ,GLM} for further reading.
\item There is a rich literature on convergence of semiclassical operators to tight-binding models, see e.g. \cite{Ha,HS1,Si2,HS2,Ma1,Ou,Ma2,Ca} and more recently \cite{FLW6,FW3}. The bulk-edge correspondence holds in many tight-binding situations, see e.g. \cite{KRS,EGS,ASV,GP,PS,Sha,Br,GS,GT,ST} and references therein. A novel approach could consist in combining these two classes of results.
\item Under certain conditions, biperiodic backgrounds can support non-linear surface solitons \cite{DP,DPR2,DNPW}. Our work suggests that there may be a topological interpretation of these results.
\item We speculate that some of the results presented here have analogs in random situations where averaging effect arise -- such as in \cite{BZ1,BZ2,BGu1,BG2,BG3,Dr4}. The quantum Hall effect has been studied in continuous random situations \cite{KS,CG,HS1,HS2,DHS,Taa}. For us the discussion is still speculative: it is not clear what random model to consider.
\item It would be nice to present a K-theoretic interpretation of our result. We point to \cite{BR} for a K-theory proof of the bulk-edge correspondence in \textit{continuous} lattice and related references. 
\end{itemize}

\subsection*{Notations}
\begin{itemize}
\item $\Ss^1 \subset \C$ is the sphere $\{z \in \C : |z|=1\}$; 
\item $\Bb(z,r) \subset \C$ denotes the ball centered at $z \in \C$, of radius $r$. 
\item If $E,E' \subset \R^N$, $\dist(E,E')$ denotes the Euclidean distance between sets. 
\item $D_x$ is the operator $\frac{1}{i}\p_x$.
\item The space $L^2_\xi$ consists of $\xi$-quasiperiodic functions (w.r.t. $\Z$):
\begin{equation}
L^2_\xi \de \left\{ u\in L^2_\loc : u(x+1) = e^{i\xi} u(x) \right\}.
\end{equation}
It splits orthogonally as $L^2_\xi = L^2_{\xi,\ev} \oplus L^2_{\xi,\od}$,
\begin{equations}
L^2_{\xi,\ev} \de \left\{ u \in L^2_\loc : u(x+1/2,\xi) = e^{i\xi/2} u(x)\right\}, 
\\
 L^2_{\xi,\od} \de \left\{ u \in L^2_\loc : u(x+1/2,\xi) = -e^{i\xi/2} u(x)\right\}.
\end{equations}
\item If $f, g$ are square integrable on $(0,1)$, we set
\begin{equation}
\lr{f,g} \de \int_0^1 \ove{f(x)} g(x)dx.
\end{equation}
\item $V$, $W$ denote smooth, real-valued potentials such that $V(x+1/2) = V(x)$ and $W(x+1/2) = -W(x)$. 
\item $W_t$ is the periodic potential $x \mapsto W(x+t/(2\pi))$.
\item $n$ is an odd integer referring to the $n$-th gap of $D_x^2 + V+W$.
\item $(\pi,E_\star)$ denotes the $n$-th Dirac point of $D_x^2+V$ and a Dirac eigenbasis is $(\phi_+^\star,\phi_-^\star) \in L^2_{\pi,\ev} \times L^2_{\pi,\od}$. See \S\ref{sec:2.4}.
\item The Fermi velocity is $\nu_\star = 2\blr{D_x\phi_+^\star,\phi_+^\star}$; we will also use $\nu_F = |\nu_\star|$.
\item $\var$ is the $2\pi$-periodic function $t \mapsto \blr{\phi_-^\star,W_t\phi_+^\star}$; $\var_\star$ equals $\var(\pi)$; the integer $m$ is the winding number of $\var$ around $0$. We will write $|\var_\star|=\var_F$
\item $P_{s}(\xi,t)$ is the operator $D_x^2 + V + sW_t$ acting on $L^2_{\xi}$. Its discrete spectrum is $\lambda_{s,1}(\xi,t) \leq \dots \leq \lambda_{s,\ell}(\xi,t) \leq \dots$.
\item $\chi_\pm$ are two smooth functions on $\R$ such that
\begin{equation}\label{eq:9p}
\chi_+ + \chi_- = 1, \ \ \ \ \chi_+(x) = \systeme{1 & \text{for } x \gg 1 \\ 0 & \text{for } x \ll -1}, \ \ \ \ \chi_-(x) = \systeme{1 & \text{for } x \ll -1 \\ 0 & \text{for } x \gg 1}.
\end{equation}
\item The Pauli matrices are
\begin{equation}\label{eq:9z}
\bsu = \matrice{0 & 1 \\ 1 & 0}, \ \ \bsd = \matrice{0 & i \\ -i & 0}, \ \  \bst = \matrice{1 & 0 \\ 0 & -1}.
\end{equation}
These matrices satisfy $\bsj^2 = \Id_2$ and $\bsi \bsj = -\bsj \bsi$ for $i \neq j$.
\item The matrix $\bss$ is defined by
\begin{equation}
\bss = \matrice{0 & \ove{\var_\star} \ \\ \var_\star & 0} = \Re(\var_\star) \bsu + \Im(\var_\star) \bsd. 
\end{equation}
\item $\Tt^1$ is the torus $\R/(2\pi\Z)$; $\Tt^2$ is the two-torus $\R^2/(2\pi\Z)^2$.
\item If $E \rightarrow \Tt^2$ is a smooth vector bundle, $\Gamma(\Tt^2,E)$ denotes the space of smooth sections of $E$.
\item $\CCC$ is the cylinder $\Tt^1 \times \R$.
\item The spectral flow of a $2\pi$-periodic family of operators $A(t)$ is $\Sf_{\Tt^1}(A)$ (see \S\ref{sec:4.0}).
\item The (first) Chern class of a vector bundle $\EE$ is $c_1(\EE)$ (see \S\ref{sec:8.3}).
\item If $\HH$ is a Hilbert space and $\psi \in \HH$, we write $|\psi|_\HH$ for the norm of $\HH$; if $A : \HH \rightarrow \HH$ is a bounded operator of $\HH$, the operator norm of $A$ is
\begin{equation}
\| A \|_\HH \de \sup_{|\psi|_\HH=1} |A\psi|_\HH.
\end{equation}
\item If $A$ is a (possibly unbounded) operator on a Hilbert space $\HH$, the spectrum of $A$ is $\Sigma_\HH(A)$. If $A$ is trace class, then $\Tr_\HH[A]$ denotes the trace of $A$.
\item If $\psi_\epsi \in \HH$ (resp. $A_\epsi : \HH \rightarrow \HH$ is a linear operator) depend on a parameter $\epsi$, we write $\psi_\epsi = O_\HH(f(\epsi))$ -- resp. $A_\epsi = \OO_\HH(f(\epsi))$ -- when there exists $C > 0$ such that $|\psi_\epsi|_\HH \leq Cf(\epsi)$ -- resp. $\|A_\epsi\|_\HH \leq C f(\epsi)$ -- for sufficiently small $\epsi$.
\end{itemize}

\subsection*{Acknowledgments} I would like to thank Alexander Watson and Michael Weinstein for introducing me to the topological aspects of condensed matter physics and for interesting discussions; Semyon Dyatlov for suggesting a more conceptual proof of Lemma \ref{lem:1b}; and Jacob Shapiro for referring me to the Kitaev table. I gratefully acknowledge support from the Simons Foundation through M. Weinstein's Math+X investigator award $\#$376319 and support from NSF DMS-1800086.

\section{Overview of Floquet--Bloch theory and 1D Dirac points}\label{sec:3.0}

\subsection{Dispersion curves} We review basic 1D Floquet--Bloch theory.  The basic reference is Reed--Simon \cite[\S XIII]{RS}. Kuchment recently wrote a nice survey \cite{Ku}.

Let $\VVV \in C^\infty(\R,\R)$ be one-periodic. The operator $D_x^2 + \VVV$ acting on the space
\begin{equation}
L^2_\xi = \left\{ u \in L^2_\loc : u(x+1) = e^{i\xi} u(x) \right\}
\end{equation}
has discrete spectrum $E_1(\xi) \leq E_2(\xi) \leq \dots$. The maps $\xi \in [0,2\pi] \mapsto E_j(\xi)$ are called dispersion curves. They are symmetric about $\xi = \pi$ (because $\VVV$ is real valued). 
The eigenvalue problem is a second order differential equation. Its solution space is at most two-dimensional. Since $\VVV$ is real-valued, one can deduce -- see e.g. \cite[Theorem XIII.89]{RS} -- that
\begin{equations}\label{eq:3h}
n \text{ odd } \Rightarrow \ E_n \text{ is increasing on } [0,\pi] \text{ and decreasing on } [\pi,2\pi]; 
\\
n \text{ even } \Rightarrow \ E_n \text{ is decreasing on } [0,\pi] \text{ and increasing on } [\pi,2\pi]. 
\end{equations}
Another consequence is that degenerate eigenvalues can occur only if $\xi = \pi \mod 2\pi$:
\begin{equation}\label{eq:4a}
E_n(\xi) = E_{n+1}(\xi) \ \Rightarrow \ \systeme{ \xi = 0 \text{ if } n \text{ is even;} \\ \xi = \pi \text{ if } n \text{ is odd.} }
\end{equation}
The next result implies that dispersion curves are stable. It is a direct consequence of \cite[Appendix A]{FW2}.

\begin{lem}\cite[Appendix A]{FW2}\label{lem:1m} Let $\ell$ be an integer and $C > 0$; there exists $\epsi_0 > 0$ such that the following holds. Let $\VVV$ and $\WWW$ be two periodic potentials and $E_{\VVV,\ell}(\xi)$ be the $\ell$-th $L^2_\xi$-eigenvalue of $D_x^2+\VVV$; $E_{\WWW,\ell}(\xi)$ be the $\ell$-th $L^2_\xi$-eigenvalue of $D_x^2+\WWW$. Then
\begin{equation}
|\VVV-\WWW|_\infty  \leq \epsi_0, \ \ |\VVV|_\infty+|\WWW|_\infty  \leq C \ \ \Rightarrow \ \ \big|E_{\WWW,\ell}(\zeta) - E_{\VVV,\ell}(\xi)\big|\leq |\VVV-\WWW|_\infty.
\end{equation}
\end{lem}

\subsection{1D Dirac points and dimer symmetry}\label{sec:2.4}

We now review 1D Dirac points. 

\begin{definition}\label{def:1}
We say that $D_x^2 + \VVV$ has a Dirac point at $(\xi_\star,E_\star)$ if there exist an integer $n$ and $\nu_F > 0$ such that
\begin{equations}\label{eq:3r}
E_n(\xi) = E_\star - \nu_F\cdot |\xi-\xi_\star| + O(\xi-\xi_\star)^2, \\ E_{n+1}(\xi) = E_\star + \nu_F \cdot |\xi-\xi_\star| + O(\xi-\xi_\star)^2.
\end{equations}
\end{definition}

In particular, $n$ is odd because of the monotonicity properties of dispersion curves  -- see \eqref{eq:3h}. Dirac points energies are two-fold degenerate eigenvalues of $D_x^2+\VVV$ on $L^2_{\xi_\star}$. 

Dirac points are usually studied in higher-dimensional lattices. Regarding mathematical proofs of their existence in honeycomb lattices, we refer to \cite{CV,BC,Le} and especially \cite{FW} which shows that they arise generically. We also mention that different types of degeneracies may arise in other types of lattices. See e.g. the Lieb lattice where dispersion surfaces intersect quadratically \cite{KMO}. Here, we will be specifically interested in Dirac points generated by \textit{dimer} symmetries. 

\begin{lem}\cite[Appendix B.1]{Wat}\label{lem:1u} Assume that 
\begin{equation}\label{eq:9w}
\VV \in C^\infty(\R,\R),  \ \ \ \ 
\VV(x+1/2) = \VV(x).
\end{equation}
Then the Dirac points of $D_x^2 + V$ (regarded as a one-periodic operator) with $\xi_\star = \pi$ are precisely the pairs $(\pi, E_n(\pi))$  where $n$ spans the odd integers.
\end{lem}

Watson \cite[Appendix B.1]{Wat} gave a proof of Lemma \ref{lem:1u}, improving upon \cite[Appendix D]{FLTW2} that showed genericity. We provide an elementary proof in Appendix \ref{sec:7.5}. These proofs use that for $\xi \in (0,2\pi)$, $D_x^2 + \VVV$ acting on the spaces
\begin{equations}
L^2_{\xi,\ev} \de \left\{ u \in L^2_\loc : u(x+1/2,\xi) = e^{i\xi/2} u(x)\right\}, 
\\
 L^2_{\xi,\od} \de \left\{ u \in L^2_\loc : u(x+1/2,\xi/2) = -e^{i\xi} u(x)\right\},
\end{equations}
have simple eigenvalues -- see Lemma \ref{lem:2m}. This fact is also the basis for:

\begin{lem}\label{lem:1c} Let $\VVV$ be $1/2$-periodic, $n$ an odd integer and $(\pi,E_\star) = (\pi,E_n(\pi))$ be a Dirac point of $D_x^2 + \VVV$. Then for each $\xi \in (0,2\pi)$, there exist $L^2_\xi$-normalized eigenpairs $(\lambda_+(\xi),\phi_+(\xi))$ and $(\lambda_-(\xi),\phi_-(\xi))$ of $D_x^2+\VVV$ varying smoothly with $\xi$ and satisfying:
\begin{itemize}
\item $\phi_+(\xi) \in L^2_{\xi,\ev}$ and $\phi_-(\xi) \in L^2_{\xi,\od}$.
\item If $\phi_+^\star \de \phi_+(\pi)$ and $\phi_-^\star \de \phi_-(\pi)$ then $\ove{\phi_+^\star} = \phi_-^\star$.
\item For $\xi$ in a neighborhood of $\pi$,
\begin{equation}
\begin{matrix}
\lambda_+(\xi) = E_\star + \nu_\star \cdot (\xi-\pi) + O(\xi-\pi)^2,  \\
\lambda_-(\xi) = E_\star - \nu_\star \cdot (\xi-\pi) + O(\xi-\pi)^2,
\end{matrix} \ \ \ \ \ \nu_\star \de 2\blr{\phi_+^\star,D_x\phi_+^\star} = - 2\blr{\phi_-^\star,D_x\phi_-^\star} \neq 0.
\end{equation}
\end{itemize}
\end{lem}

We refer to \cite[Proposition 3.7]{FLTW1} for a more general statement and to \cite[Proposition 2.1]{DFW} for the statement and proof of the version needed here. The numbers $\nu_\star$ and $\nu_F$ are related via $|\nu_F| = \nu_\star$. As a consequence of the monotonicity property of the eigenvalues of $D_x^2+V$ on $L^2_{\xi,\ev}$ and $L^2_{\xi,\od}$ -- see Lemma \ref{lem:2m} and \eqref{eq:6q} -- the sign of $\nu_\star$ is related to the value of $n$ modulo $4$:
\begin{equation}\label{eq:9t}
n = 1 \mod 4 \ \Rightarrow \ \nu_\star > 0, \ \ \ \ n = 3 \mod 4 \ \Rightarrow \ \nu_\star < 0.
\end{equation}
In particular, $\sgn(\nu_\star) = (-1)^{\frac{n-1}{2}}$. 

The pair $(\phi_+^\star,\phi_-^\star)$ solves
\begin{equations}\label{eq:3s}
\big( D_x^2+\VVV-E_\star \big)\phi_\pm^\star = 0, \ \ \phi_+^\star(x) = \ove{\phi_-^\star(x)},  \\ 
\phi_+^\star \in L^2_{\pi,\ev}, \ \ \phi_-^\star \in L^2_{\pi,\od}, \ \ |\phi_+|_{L^2_\pi} =  |\phi_-|_{L^2_\pi} = 1,
\end{equations}
Because the eigenvalues of $D_x^2+\VVV$ on $L^2_{\pi,\ev}$ and $L^2_{\pi,\od}$ are simple, this characterizes $(\phi_+^\star,\phi_-^\star)$ uniquely, modulo the $\Ss^1$-action $(\phi_+^\star,\phi_-^\star) \mapsto (\w\phi_+^\star,\ove{\w}\phi_-^\star)$ where $\w \in \Ss^1$.
We refer to $(\phi_+^\star,\phi_-^\star)$ as a Dirac eigenbasis. 

We will be interested in bifurcations of Dirac points when the $1/2$-translational symmetry is broken. For the rest of the paper, we fix one-periodic potentials $V$ and $W$ with
\begin{equation}
V, \ W \in C^\infty(\R,\R),  \ \ \ \ V(x+1/2) = V(x), \ \ \ \
W(x+1/2) = -W(x).
\end{equation}
Fix $n$ an odd integer. Because $V$ satisfies the assumption of Lemma \ref{lem:1u}, it has a Dirac point at $(\pi,E_\star)$, where $E_\star$ is the $n$-th eigenvalue of $D_x^2 + V$ on $L^2_\pi$. Let $\nu_\star, \ \phi_+(\xi), \ \phi_-(\xi)$ be the objects associated via Lemma \ref{lem:1c}.

We define $W_t$ as $W_t(x) \de W\big(x+t/(2\pi) \big)$. Introduce
\begin{equation}\label{eq:5d}
\var(t) = \blr{\phi_-^\star,W_t \phi_+^\star}.
\end{equation}
If $\var(t)$ never vanishes -- assumption $\oB$ -- then $\var$ has a well-defined winding number on $[0,2\pi]$. Because of the above characterization of $(\phi_+^\star,\phi_-^\star)$ modulo the $\Ss^1$-action, this winding number is independent of the choice of Dirac eigenbasis. It is given by
\begin{equation}
m = \dfrac{1}{2\pi i} \int_0^{2\pi} \dfrac{\var'(t)}{\var(t)} dt.
\end{equation}
We claim that $m$ is odd. Indeed, the winding number of $\tvartheta : t \mapsto e^{-it} \var(t)$ on $[0,2\pi]$ is $\tm = m-1$. Moreover $\tvartheta$ satisfies $\tvartheta(t+\pi) = \tvartheta(t)$; therefore $\tvartheta$ is $\pi$-periodic and its winding number on $[0,2\pi]$ is even. This proves that $m$ is odd. We conclude this section with a simple lemma, proved in \cite[Proposition 2.2]{DFW} when $t=\pi$. The proof goes through without changes when $t \neq \pi$.

\begin{lem}\cite[Proposition 2.2]{DFW}\label{lem:2z} With the above notations,
\begin{equations}
\matrice{\blr{\phi_+^\star,W_t \phi_+^\star} & \blr{\phi_+^\star,W_t \phi_-^\star} \\ \blr{\phi_-^\star,W_t \phi_+^\star} & \blr{\phi_-^\star,W_t \phi_-^\star}} = \matrice{0 & \ove{\var(t)} \ \\ \var(t) & 0}. 
\end{equations}
\end{lem}

\section{Examples}\label{sec:8}

In this section, we present settings where Theorems \ref{thm:1} and \ref{thm:2} apply. We start with pairs $(V,W)$ that display the same topology as the dislocation problem of \cite{Ko1,DPR,HK}. Let $\VV, \WW$ be one-periodic potentials with:
\begin{equation}
\VV,\ \WW \in C^\infty(\R,\R), \ \ \ \ \VV(x+1/2) = \VV(x), \ \ \ \ \WW(x+1/2) = -\WW(x+1/2).
\end{equation}
If $f \in C^\infty(\R,\C)$ is a one-periodic function we denote by $\hf_\ell$ its $\ell$-th Fourier coefficient:
\begin{equation}
\hf_\ell \de \int_0^1 f(x)e^{-2i\pi \ell x} dx.
\end{equation}

\begin{lem}\label{lem:1f} Fix $n$ an odd integer. Assume that $\hWW_n \neq 0$; and $n=1$ or $\hVV_{n-1} \neq 0$. If $\epsi$ is sufficiently small then the pair $(V,W) = (\epsi \VV, \epsi \WW)$ satisfies \eqref{eq:9v},  $\oA$ and $\oB$. In addition, the winding number \eqref{eq:1z} for the $n$-th gap equals $(-1)^{\frac{n+1}{2}}  \cdot n$ and the bulk/edge index equals $-n$.
\end{lem}

This shows that small potentials $V, \ W$ exhibit the same topological features as the dislocation problem of \cite{Ko1,DPR,HK}: $n$ eigenvalues effectively flow in the $n$-th gap as $t$ goes from $0$ to $2\pi$.\footnote{The sign discrepancy in comparison with \cite{HK} comes from the fact that the dislocation arises on the left there, while here we consider dislocations on the right.} The proof of Lemma \ref{lem:1f} is a perturbative argument. We postpone it to Appendix \ref{app:7}.

We now present a topologically different setting: we construct a pair $(V,W)$ of potentials satisfying the hypothesis $\oA$ and $\oB$ such that the winding number \eqref{eq:1z} in the $n$-th gap is not equal to $n$.

\begin{lem}\label{lem:2n} Let $n, m$ be two odd integers with $n > 0$ and $|m| \neq n$. Define
\begin{equation}
V(x) \de \epsi^2 \cos\big(2\pi(n-m)x\big) + \epsi^3 \cos\big(2\pi (n-1) x\big), \ \ W(x) \de 2\epsi^4 \cos(2 \pi m x).
\end{equation}
If $\epsi$ is sufficiently small then the pair $(V,W)$ satisfies \eqref{eq:9v}, $\oA$ and $\oB$. Furthermore, the winding number \eqref{eq:1z} for the $n$-th gap equals $(-1)^{\frac{n+1}{2}} \cdot m$ and the bulk/edge index equals $-m$.
\end{lem}

This result contrasts fundamentally with the dislocation problem studied in \cite{Ko1,DPR,HK}. There, $V \equiv 0$: $n$ eigenvalues effectively flow in the $n$-th gap. Here, $|m|$ eigenvalues effectively flow in the $n$-th gap. This demonstrates that our class of operators is topologically richer than that of \cite{Ko1,DPR,HK} -- hence Theorem \ref{thm:1} \textit{needs} some assumptions to hold. From the physics point of view, this means that the periodic background (modeled by $V$) may fundamentally affect topological phases of matter. We emphasize that this example is not exceptional: small perturbations of $(V,W)$ do not change the winding number \eqref{eq:1z}. Lemma \ref{lem:2n} is somewhat reminiscent of \cite{KM}.

The proof of Lemma \ref{lem:1f} is also a perturbative argument. We postpone it to Appendix \ref{app:7}. We conclude with the following conjecture:

\begin{conj} Let $m_1, m_3, \dots$ be an infinite sequence of odd integers, labeled by odd positive integers. For every odd $n$, there exist $(V,W)$ satisfying \eqref{eq:9v}, $\oA$ and $\oB$ for the $n$-th gap, with winding number \eqref{eq:1z} equal to $(-1)^{\frac{n-1}{2}} \cdot m_n$.
\end{conj}

\section{Reduction to a Dirac operator and proof of Theorem \ref{thm:1}}\label{sec:2}

We prove here Theorem \ref{thm:1}. We briefly recall the context. We analyze the family
\begin{equation}
\PP(t) = D_x^2 + V + \chi_+ W + \chi_- W_t + F, \ \ \ \ t\in [0,2\pi].
\end{equation}
Above, $V$ and $W$ are smooth, real-valued periodic potentials satisfying \eqref{eq:9v}; $W_t(x) = W(x+t/(2\pi))$; $\chi_\pm$ satisfy \eqref{eq:9p}; and $F$ is a compact operator  from $H^2$ to $L^2$. We fix $n$ an odd integer and let $(\pi,E_\star)$ be the $n$-th Dirac point of the unperturbed operator $D_x^2+V$ -- see \S\ref{sec:2.4}. This comes with a Dirac eigenbasis, i.e. a normalized pair $(\phi_+^\star,\phi_-^\star)$ that solves the eigenvalue problem
\begin{equation}
\big(D_x^2 + V - E_\star\big) \phi_\pm^\star = 0, \ \ \ \ \phi_\pm^\star \in L^2_{\pi,\ev}, \ \ \phi_-^\star \in L^2_{\pi,\od}.
\end{equation}

We review the assumptions $\oA$ and $\oB$. For every $t$, the $n$-th spectral gap of $D_x^2 + V + W_t$ is open. When the strength of the dislocation decreases to $0$, this gap persists and shrinks to $\{E_\star\}$  -- see $\oA$. The non-degeneracy condition 
\begin{equation}
t \in [0,2\pi] \ \ \Rightarrow \ \ 
\var(t) \de \blr{\phi_-^\star,W_t\phi_+^\star} \neq 0
\end{equation}
holds -- see $\oB$. Theorem \ref{thm:1} asserts that 
\begin{equation}\label{eq:9r}
\Sf_{\Tt^1} \big(\PP-E(1,\cdot) \big) = (-1)^{\frac{n-1}{2}} \cdot \dfrac{1}{2\pi} \int_0^{2\pi} \dfrac{\var'(t)}{\var(t)} dt.
\end{equation}

We organize the proof as follows. We review the notion of spectral flow in \S\ref{sec:4.0}.  In \S\ref{sec:3.3a}, we deform the family $t \mapsto \PP(t)-E(1,t)$ to a family $t \mapsto \PP_\delta(t)-E_\star$ exhibiting a \textit{small and adiabatic} dislocation instead. $\oA$ ensures that the essential spectral gap remains open. The operator $\PP_\delta(\pi)$ has been studied in \cite{FLTW1,FLTW2,DFW}. In \cite{DFW}, we show that the eigenvalues of $\PP_\delta(t)$ in the gap containing $E_\star$ are \textit{all precisely} seeded by those of an effective Dirac operator,
\begin{equation}
\DDi_\star \de \nu_\star \bst  D_x -   \chi_-\bss +\chi_+ \matrice{0 & \ove{\var_\star} \ \\ \var_\star & 0}, \ \ \ \ \nu_\star \de \blr{\phi_+^\star,D_x\phi_-^\star} \neq 0.
\end{equation}
The analysis applies similarly to $\PP_\delta(t)$ and produces a Dirac operator $\DDi(t)$. A consequence is that $\PP-E(1,\cdot)$ and $\DDi$ have the same spectral flow -- see \S\ref{sec:3.3}.

In other words, we found a gap-preserving homotopy that starts at $t \mapsto \PP(t)-E(1,t)$ -- these are second order  differential operators, periodic in the bulk -- and ends at $t \mapsto \DDi(t)$ -- these are \textit{first order} differential operators \textit{constant} in the bulk. The integers $\Sf_{\Tt^1}(\DDi)$ and $\Sf_{\Tt^1}(\PP-E(1,\cdot))$ are equal because the spectral flow is invariant along gap-preserving homotopies.

 Because the coefficients of $\{\DDi(t)\}_{t\in [0,2\pi]}$ are constant in the bulk, a simple computation of its spectral flow is possible.
The winding number $m$ displayed in the RHS of \eqref{eq:9r} emerges via another homotopy that transforms $\var(t)$ to $-\var_\star e^{imt}$ -- see \S\ref{sec:3.5b}. In \S\ref{sec:3.6a}, we calculate explicitly the spectrum of a Dirac operator with piecewise constant coefficients. It leads to the value of $\Sf_{\Tt^1}(\DDi)$, completing the proof of Theorem \ref{thm:1}.

\subsection{Spectral flow}\label{sec:4.0} We review the definition of spectral flow, following \cite[\S4]{Wa}. Let $\HH$ be a Hilbert space and $\{T(t)\}_{t\in [0,2\pi]}$ be a family of selfadjoint Fredholm operators on $\HH$ with $t$-independent domains, continuous for the gap distance
\begin{equation}
d(T,T') \de \left\| (T+i)^{-1} - (T'+i)^{-1} \right\|_\HH.
\end{equation}
Assume that $T(0) = T(2\pi)$; and that $0$ is not in the essential spectrum of $T(t)$ for every $t\in [0,2\pi]$. There exist $0 = t_1 < \dots < t_N = 2\pi$ and $a_1, \dots, a_N > 0$ such that
\begin{equation}\label{eq:1s}
t \in [t_j,t_{j+1}] \ \Rightarrow \ \pm a_j \notin \Sigma_\HH\big(T(t)\big) \text{ and } \1_{[-a_j,a_j]}\big(T(t)\big) \text{ has finite rank.}
\end{equation}
This in particular implies that for $t \in [t_j,t_{j+1}]$, $\1_{[-a_j,a_j]}\big(T(t)\big)$ has constant rank. Figure \ref{fig:11} is a picture representing \eqref{eq:1s} on an example.
The spectral flow is
\begin{equation}
\Sf_{\Tt^1}(T) \de \sum_{j=1}^{N-1} \dim \1_{[0,a_j]}\big(T(t_{j+1})\big) - \dim \1_{[0,a_j]}\big(T(t_j)\big).
\end{equation}
It is independent of the choice of $t_j$, $a_j$ that satisfy \eqref{eq:1s}, see \cite[Lemma 4.1.3]{Wa}.

\begin{center}
\begin{figure}
{\caption{The gray region represents the essential spectrum of a family of operators $\{T(t)\}$. The red dotted curves represent its eigenvalues in the gap. We partition $[0,2\pi]$ in sub-intervals so that $T(t)$ has a constant number of eigenvalues in the gap, in each sub-interval. Above, this number is $0$ in $[0,t_1]$; $1$ in $[t_1,t_2]$; and $0$ in $[t_2,2\pi]$. The spectral flow~is~$1$.}\label{fig:11}}
{\includegraphics{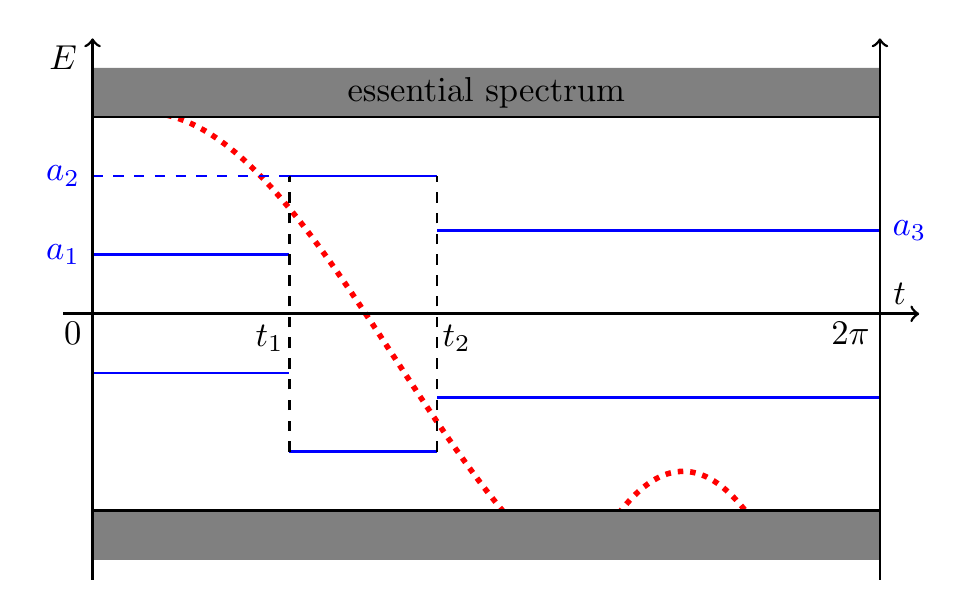}}
\end{figure}
\end{center}
\vspace*{-.9cm}

The spectral flow is invariant under homotopy. Specifically, assume given for every $s \in [0,1]$ a periodic family $\{T_s(t)\}_{t\in [0,2\pi]}$ of selfadjoint Fredholm operators on $\HH$ with domains independent on $(s,t)$, such that $0$ is not in the essential spectrum of $T_s(t)$. If $(s,t) \in [0,1] \times [0,2\pi] \mapsto T_s(t)$ is continuous for the gap distance then 
\begin{equation}
s \in [0,1] \mapsto \Sf_{\Tt^1}(T_s)
\end{equation}
is constant -- see \cite[Theorem 4.2.4]{Wa}.

\subsection{Homotopy}\label{sec:3.3a} We start the proof of Theorem \ref{thm:1}. For the rest of \S\ref{sec:2}, we fix functions $V$, $W$, $\chi_\pm$, $F$ satisfying the setting of \S\ref{sec:1.1}, together with an odd positive integer $n$. We assume that the pair $(V,W)$ satisfies $\oA$ and $\oB$ for the $n$-th gap. 

We want to compute the edge index: $\Sf_{\Tt^1}(\PP-E(1,\cdot))$. For $s \in (0,1]$, we introduce
\begin{equation}
\PP_s(t) \de D_x^2 + V + s \chi_-(s \ \cdot) W + s \chi_+(s \ \cdot) W_t +\te(s) F
\end{equation}
where $\te : [0,1] \rightarrow \R$ is a smooth function with value $0$ near $0$ and $1$ near $1$; and $\chi_\pm(s \ \cdot) : x \mapsto \chi_\pm(sx)$. By assumption $\oA$, $E(s,t)$ is not in the essential spectrum of $\PP_s(t)$ for any $(s,t) \in (0,1] \times [0,2\pi]$.

Hence $\Sf_{\Tt^1}(\PP-E(1,\cdot)) = \Sf_{\Tt^1}(\PP_s-E(s,\cdot))$ as long as the family 
\begin{equation}\label{eq:1r}
(s,t) \in (0,1] \times \R \mapsto (\PP_s(t)-E(s,t)+i)^{-1} : \ L^2 \rightarrow L^2
\end{equation}
is continuous for the gap distance. To verify this, we use the resolvent identity:
\begin{equations}
\big(\PP_s(t)-E(s,t)+i\big)^{-1} - \big(\PP_{s'}(t')-E(s',t')+i\big)^{-1}
\\
= \big(\PP_s(t)-E(s,t)+i\big)^{-1} \cdot \big(\PP_s(t) - \PP_{s'}(t') + E(s',t')-E(s,t) \big) \cdot \big( \PP_{s'}(t')-E(s',t')+i \big)^{-1}.
\end{equations}
The map $(s,t) \mapsto E_s(t)$ is continuous and $\|\PP_s(t) - \PP_{s'}(t')\|_{L^2}$ goes to $0$ when $(s',t')$ goes to $(s,t)$. Because of the spectral theorem, 
\begin{equation}
\left\|\big(\PP_s(t)-E(s,t)+i\big)^{-1}  \right\|_{L^2} \leq 1, \ \ \ \ \left\|\big( \PP_{s'}(t')-E(s',t')+i \big)^{-1} \right\|_{L^2} \leq 1.
\end{equation}
Therefore the map \eqref{eq:1r} is continuous for the gap distance. Hence
\begin{equation}\label{eq:2k}
\Sf_{\Tt^1}\big(\PP-E(1,\cdot)\big) = \Sf_{\Tt^1}\big(\PP_\delta-E(\delta,\cdot)\big) = \Sf_{\Tt^1}\big(\PP_\delta-E_\star\big).
\end{equation}
In the last equality, we used $E(\delta,\cdot) = E_\star$ when  $\delta$ is sufficiently small.
In contrast with $\PP$, the dislocation in $\PP_\delta$ is small and adiabatic. This is the setting of the papers \cite{FLTW1,FLTW2,DFW}.

\begin{rmk}\label{rem:1} There exists $\epsi_0 > 0$ such that the identity \eqref{eq:2k} persist if $F$ is replaced by any selfadjoint family $\{F_1(t)+F_2(t)\}$, where
\begin{itemize}
\item $F_1(t)$ and $F_2(t)$ are bounded $H^2 \rightarrow L^2$, continuous and $2\pi$-periodic in $t$.
\item For every $t$, $F_1(t) : H^2 \rightarrow L^2$ with $\|F_1(t) \|_{L^2 \rightarrow H^2} \leq \epsi_0$;
\item For every $t$, $F_2(t) : H^2 \rightarrow L^2$ is compact.
\end{itemize}
In particular, Theorem \ref{thm:1} holds for a wider class of perturbations than presented in \S\ref{sec:1.1}. We leave this verification to the reader.\end{rmk}

\subsection{Review of \cite{DFW}}\label{sec:3.3} We recall that $V, W$ satisfy \eqref{eq:9v} and that $\oB$ holds: $\var(t) \neq 0$ for every $t \in \R$. The paper \cite{DFW} studies defect states of the operator
\begin{equation}
\PP_{\delta,\star} = D_x^2 + V + \delta \chi_-(\delta\cdot) W - \delta \chi_+(\delta\cdot)  W.
\end{equation}
Let $(\pi,E_\star)$ be a Dirac point of $D_x^2+V$. Assume that the number $\var_\star = \var(\pi) = \lr{\phi_-^\star,W\phi_+^\star}$ associated to $(\pi,E_\star)$ is non-zero. A multiscale approach due to \cite{FLTW2} produces the Dirac operator 
\begin{equation}
\DDi_\star = \nu_\star \bst  D_x + \bss \kappa, \ \ \ \ \bss \de \matrice{0 & \ove{\var_\star} \ \\ \var_\star & 0}, \ \ \kappa \de \chi_+-\chi_-.
\end{equation}
This operator has essential spectrum outside $(-\var_F,\var_F)$, where $var_F = |\var_\star|$. In this gap, $\DDi_\star$ has an odd number of simple eigenvalues $\{ \mu_j \}_{j \in [0, 2M_\star]}$ with
\begin{equation}
-\var_F < \mu_0 < \dots < \mu_{M_\star-1} < \mu_{M_\star} = 0 < \mu_{M_\star+1} < \dots < \mu_{2M_\star} < \var_F,
\end{equation}
see \cite[Theorem 3]{DFW}. The number $M_\star$ can be arbitrarily large; see \cite{LWW}. Zero is a robust eigenvalue of $\DDi_\star$: it persists even when $\kappa$ is modified locally. 
The main result of \cite{DFW} is:

\begin{theorem}\cite[Theorem 1 and Corollary 1]{DFW}\label{thm:4} For any $\var_\sharp \in (\var_{2M_\star},\var_F)$, there exists $\delta_0 > 0$ such that for $\delta \in (0,\delta_0)$, the following holds. The spectrum of $\PP_{\delta,\star}$ in $(E_\star-\var_\sharp \delta, E_\star + \var_\sharp \delta)$ consists exactly of $2M_\star+1$ eigenvalues, given by
\begin{equation}
E_\star + \delta \mu_j + O\left(\delta^2\right).
\end{equation}
\end{theorem}

See Figure \ref{fig:2}.
This result refines \cite{FLTW2}, which predicted existence of defect states for $\PP_{\delta,\star}$, associated to the eigenvalues of $\DDi_\star$. The main improvement of \cite{DFW} over \cite{FLTW2} is the full identification of eigenvalues of $\PP_{\delta,\star}$ in the gap $(-\var_\sharp,\var_\sharp)$. The perturbative strategy of \cite{FLTW2} could not achieve this characterization, which we will strongly need here.\footnote{The work \cite{DFW} also includes \textit{full rigorous} expansions of eigenpairs in powers of $\delta$, whose terms are constructed recursively via the multiscale approach of \cite{FLTW2}.}

\begin{center}
\begin{figure}
{\caption{The spectrum of $\DDi_\star$ and that of $\PP_{\delta,\star}$ are in bijection via a function $\mu \mapsto E_\star+\delta \mu$ modulo $O(\delta^2)$. The grey region represents the essential spectra, the red dots are eigenvalues of $\DDi_\star$, the blue dots are eigenvalues of $\PP_{\delta,\star}$. Theorem \ref{thm:4} does not apply in the light gray region near $\pm \var_F$, the edges of the essential spectrum of $\PP_{\delta,\star}$.}\label{fig:2}}
{\includegraphics{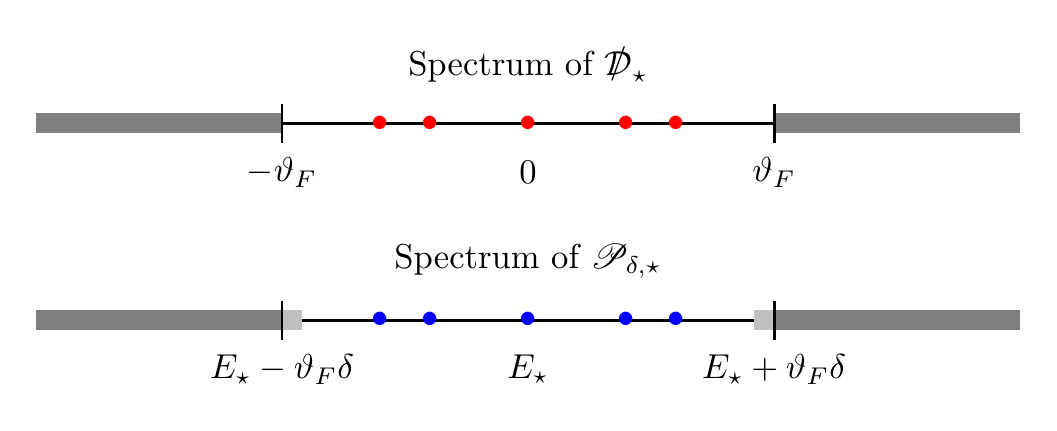}}\end{figure}
\end{center}
\vspace*{-.7cm}

We need a version of Theorem \ref{thm:4} adapted to $\PP_\delta(t)$. We assume that $\oB$ holds: $\var(t) \neq 0$ for every $t$. We introduce the Dirac operator
\begin{equation}
\DDi(t) \de \nu_\star \bst  D_x -   \chi_-\bss +\chi_+ \matrice{0 & \ove{\var(t)} \\ \var(t) & 0}.
\end{equation}
The operator $\DDi(t)$ has essential spectrum equal to 
\begin{equation}
\big(-\infty,-\var_-(t)\big] \cup \big[\var_-(t),+\infty\big), \ \ \ \ \var_-(t) \ \de \min\big(\var_F, |\var(t)| \big).
\end{equation}
Observe that  $\var_-(t)$ is positive by assumption $\oB$.
 Let $\{\mu_j(t)\}_{j \in [0, n(t)]}$ be the ordered eigenvalues of $\DDi(t)$ in $(-\var_-(t),\var_-(t))$. Note that $n(t)$ is finite.

\begin{theorem}\label{thm:3} For any $t \in \R$, $\var_\flat \in \big(-\var_-(t),\mu_0(t)\big)$ and $\var_\sharp \in \big(\mu_{n(t)}(t),\var_-(t)\big)$, there exists $\delta_0 > 0$ such that for $\delta \in (0,\delta_0)$, the following holds. The spectrum of $\PP_\delta(t)$ in $(E_\star+ \var_\flat \delta, E_\star + \var_\sharp \delta)$ consists exactly of $n(t)+1$ eigenvalues, given by
\begin{equation}
E_\star + \delta \mu_j(t) + O\left(\delta^2\right).
\end{equation}
\end{theorem}

We will not prove this result. The proof is identical to that of \cite[Theorem 1 and Corollary 1]{DFW} and we refer to that paper. Theorem \ref{thm:3} has the corollary:

\begin{cor}\label{cor:2} Fix $t \in \R$ and $\var_\flat < \var_\sharp$ outside the spectrum of $\DDi(t)$. There exists $\delta_0(t) > 0$ such that for any $\delta \in \big(0,\delta_0(t) \big)$, 
\begin{equations}
E_\star+\delta\var_\flat, \ E_\star+\var_\sharp \delta \notin \Sigma_{L^2}\big( \PP_\delta(t) \big) \ \ \  \text{ and} \\
\dim \1_{[E_\star-\delta\var_\flat,E_\star+\delta\var_\sharp]}\big(\PP_\delta(t)\big) = \dim \1_{[\var_\flat,\var_\sharp]}\left(\DDi(t)\right).
\end{equations}
\end{cor}

We use this corollary to relate the spectral flow of $\PP_\delta(t)$ to that of $\DDi(t)$. 

\begin{lem}\label{prop:1} Under the above assumptions, for every $\delta \in (0,1]$,
\begin{equation}\label{eq:9s}
\Sf_{\Tt^1}\big(\DDi\big) = \Sf_{\Tt^1}\big(\PP_\delta-E_\star\big).
\end{equation}
\end{lem}

This result is quite unusual. A first order differential operator with coefficients constant in the bulk stands in the LHS of \eqref{eq:9s}; while a second order differential operator with coefficients periodic in the bulk appears in the RHS.

\begin{proof} 1. Because of \eqref{eq:2k}, we need to prove \eqref{eq:9s} only for small $\delta$. Consider $\{t_j, \ a_j\}_{j \in [1,N]}$ satisfying \eqref{eq:1s}:
\begin{equation}
t \in [t_j,t_{j+1}] \ \Rightarrow \ \pm a_j \notin \Sigma_{L^2}\big(\DDi(t)\big) \text{ and } \1_{[-a_j,a_j]}\big(\DDi(t)\big) \text{ has finite rank.}
\end{equation}
There exists $\epsi_0 > 0$ such that the following conditions are satisfied :
\begin{itemize}
\item[(i)] For every $j \in [1,N-1]$, for every $t \in [t_j,t_{j+1}]$, the operators $\DDi(t)$ has no spectrum in $[-a_j-\epsi_0,-a_j+\epsi_0] \cup [a_j - \epsi_0, a_j+\epsi_0]$.
\item[(ii)] $\DDi(t_1), \ \dots, \ \DDi(t_N)$ have no spectrum in $[-\epsi_0,0)$.
\item[(iii)] For every $t \in [0,2\pi]$, $\epsi_0 \leq |\var_-(t)|$.
\end{itemize}
Because of (i), the collection $\{t_j, \ a_j-\epsi_0\}_{j \in [1,N]}$ satisfies \eqref{eq:1s}. Thus
\begin{equations}
\Sf_{\Tt^1}\left(\DDi\right) = \sum_{j=1}^{N-1} \dim \1_{[0,a_j-\epsi_0]}\left(\DDi(t_{j+1})\right) - \dim \1_{[0,a_j-\epsi_0]}\left(\DDi(t_j)\right).
\end{equations} 
From this equation and (ii), we deduce that
\begin{equations}
\Sf_{\Tt^1}\left(\DDi\right) = \sum_{j=1}^{N-1} \dim \1_{[-\epsi_0,a_j-\epsi_0]}\left(\DDi(t_{j+1})\right) - \dim \1_{[-\epsi_0,a_j-\epsi_0]}\left(\DDi(t_j)\right).
\end{equations} 

2. We apply Corollary \ref{cor:2} simultaneously to all $\PP_\delta(t_j)$ with $\var_\flat = -\epsi_0$ and $\var_\sharp = a_j-\epsi_0$; this is possible because $\{t_j\}_{j \in [1,N]}$ is finite. We deduce
\begin{equations}
\Sf_{\Tt^1}\left(\DDi\right) = \sum_{j=1}^{N-1} \dim \1_{[E_\star-\epsi_0\delta, E_\star+ a_j\delta-\epsi_0\delta]}\big(\PP_\delta(t_{j+1})\big) - \dim \1_{[E_\star-\epsi_0\delta, E_\star+ a_j\delta-\epsi_0\delta]}\big(\PP_\delta(t_j)\big)
\end{equations}
\begin{equations}\label{eq:9y}
= \sum_{j=1}^{N-1} \dim \1_{[0, \delta a_j]}\big(\PP_\delta(t_{j+1})-E_\star + \epsi_0 \delta\big) - \dim \1_{[0, \delta a_j]}\big(\PP_\delta(t_j)-E_\star+\epsi_0\delta\big).
\end{equations}

3. We would like to conclude that \eqref{eq:9y} is the spectral flow of $\PP_\delta-E_\star$. We first show that $\{\delta a_j, t_j\}_{j \in [1,N]}$ satisfies \eqref{eq:1s} for the family $\PP_\delta - E_\star + \epsi_0 \delta$. This only amounts to show that there exists $\delta_1 > 0$ with
\begin{equation}
\delta \in (0,\delta_1), \ \
\tau \in [t_j,t_{j+1}] \ \ \Rightarrow \ \ E_\star \pm \delta a_j - \epsi_0 \delta \notin \Sigma_{L^2}\big( \PP_\delta(\tau) \big).
\end{equation}
Observe that for every $\tau \in [t_j,t_{j+1}]$, $\var_\sharp = a_j-\epsi_0$, $\var_\flat = - a_j - \epsi_0$ satisfy the assumptions of Corollary \ref{cor:2} with $t = \tau$ because of (i).  This implies that there exists $\delta_0(\tau)$ with
\begin{equation}\label{eq:9x}
\delta \in \big(0,\delta_0(\tau) \big) \ \ \Rightarrow \ \ E_\star \pm \delta a_j - \epsi_0 \delta \notin \Sigma_{L^2}\big( \PP_\delta(\tau) \big).
\end{equation}
Because the family $\PP_\delta$ is continuous, \eqref{eq:9x} still holds on a small neighborhood of $\tau$: there exists $\eta(\tau) > 0$ such that
\begin{equation}
\delta \in \big(0,\delta_0(\tau) \big), \ \ |t-\tau| \leq \eta(\tau) \ \  \Rightarrow \ \ E_\star \pm \delta a_j - \epsi_0 \delta \notin \Sigma_{L^2}\big( \PP_\delta(\tau) \big).
\end{equation}
Now cover the compact set  $[t_j,t_{j+1}]$ by finitely many open sets $(\tau_\ell-\eta_\ell,\tau_\ell+\eta_\ell)$. Take $\delta_1$ to be the minimum of all $\delta_0(\tau_\ell)$ and conclude:
\begin{equation}
\delta \in \big(0,\delta_1 \big), \ \   \tau \in [t_j,t_{j+1}] \ \  \Rightarrow \ \ E_\star \pm \delta a_j - \epsi_0 \delta \notin \Sigma_{L^2}\big( \PP_\delta(\tau) \big).
\end{equation}
After shrinking $\delta_1$, we can assume that this holds for every $j \in [1,N-1]$.
We deduce from the definition of spectral flow and Step 2 that
\begin{equation}
\Sf_{\Tt^1}\left(\DDi\right) = \Sf_{\Tt^1}\big(\PP_\delta-E_\star+\epsi_0\delta\big).
\end{equation}
On the RHS, we have the spectral flow of $\PP_\delta$, measured at $E_\star- \epsi_0\delta$. This corresponds to the spectral flow of $\PP_\delta$ in the $n$-th gap because $\epsi_0$ satisfies (iii): the energy level $E_\star - \epsi_0 \delta$ belongs to the gap of $\PP_\delta(t)$ containing $E_\star$. Hence $\Sf_{\Tt^1}(\PP_\delta-E_\star+\epsi_0\delta)$ equals $\Sf_{\Tt^1}(\PP_\delta-E_\star)$. This concludes the proof.
\end{proof}

\subsection{Connection to the winding number}\label{sec:3.5b} We relate the winding number $m$ of $\var$ around $0$ to the spectral flow of $\DDi$. Introduce 
\begin{equation}
\Di_m(t) \de \sgn(\nu_\star)\bst  D_x - \chi_-\bss   - \chi_+ \matrice{0 & \ove{\var_\star}e^{-imt} \\ \var_\star e^{imt} & 0}.
\end{equation}
The functions $\var : \Tt^1 \rightarrow \C \setminus \{0\}$ and $t \mapsto -\var_\star e^{imt}$ have same degree. Since $\pi_1\left(\C \setminus \{0\}\right) = \Z$, there exists a continuous function $h : [0,1] \times \R \rightarrow \C \setminus \{0\}$ such that
\begin{equation}
h(1,t) = \var(t), \ \ h(0,t) = -\var_\star e^{imt}, \ \ h(s,t+2\pi) = h(s,t).
\end{equation}
Set
\begin{equation}
\DDi_s(t) \de \big(s\nu_\star+(1-s)\sgn(\nu_\star)\big) \bst  D_x - \chi_-\bss +\chi_+ \matrice{0 & \ove{h(s,t)} \ \\ h(s,t) & 0}.
\end{equation}
We observe that $\| \DDi_s(t) - \DDi_{s'}(t') \|_{L^2} \leq C|h(s,t)-h(s',t')|$ and that $\DDi_s(t)$ is Fredholm $H^1(\R) \rightarrow L^2(\R)$ for every $s \in [0,1]$.  Therefore, an argument as in \S\ref{sec:3.3a} shows that the family $s \mapsto \DDi_s(t)$ behaves continuously for the gap topology. 

This implies that that the spectral flows of $\DDi_1 = \DDi$ and $\DDi_0 = \Di_m$ are equal. Moreover, because of the concatenation properties of the spectral flow \cite[Lemma 4.2.2]{Wa}, $\Sf_{\Tt^1}(\Di_m) = m \cdot \Sf_{\Tt^1}(\Di_1)$. From Lemma \ref{prop:1}, we deduce:
\begin{equation}\label{eq:1b}
\Sf_{\Tt^1}\big(\PP_\delta-E_\star\big) = m \cdot \Sf_{\Tt^1}\big(\Di_1\big).
\end{equation}

\subsection{Computation of $\Sf_{\Tt^1}(\Di_1)$.} \label{sec:3.6a}

\begin{lem}\label{lem:5d} If $\chi_+ = \1_{[0,\infty)}$ and $\chi_- = \1_{(-\infty,0]}$ then for $t \in (0,2\pi)$, the operator 
\begin{equation}
\Di_1(t) \de \sgn(\nu_\star)\bst  D_x - \chi_-\bss   - \chi_+ \matrice{0 & \ove{\var_\star}e^{-it} \\ \var_\star e^{it} & 0}
\end{equation}
has a single eigenvalue in $(-\var_F,\var_F)$. It equals $\sgn(\nu_\star) \var_F \cdot \cos(t/2)$.
\end{lem}

\begin{proof} We first work in the case $\nu_\star > 0$. Fix $t \in (0,2\pi)$ and assume that $u \in H^1(\R)$ is a non-zero solution of $(\Di_1(t)-\lambda) u = 0$. In particular $u$ is continuous. From standard elliptic regularity theory, $u$ is smooth except possibly at $x=0$. For $x \neq 0$, $u$ solves a first order ODE with constant coefficients. Hence, there exist $a, b, \gamma_\pm \in \C$ with
\begin{equation}
u(x) = \matrice{a  \\ b }e^{i\gamma_\pm x} \ \text{ for } \ \pm x \geq 0, \ \ \pm \Im \gamma_\pm > 0, \ \ (a,b) \neq (0,0).
\end{equation}
The above constants must solve the equations
\begin{equations}\label{eq:2i}
\left(\bst \gamma_+ - \matrice{0 & \ove{\var_\star} e^{-it} \\ \var_\star e^{it} & 0} - \lambda \right)\matrice{a \\ b}  = 0, \ \ \ \ \big( \bst \gamma_- - \bss-\lambda \big) \matrice{a \\ b} = 0, \\  \pm \Im \gamma_\pm > 0, \ \ \ \ (a,b) \neq (0,0).
\end{equations}
Taking the difference of the two equations, we obtain
\begin{equation}\label{eq:2h}
\left(\bst(\gamma_+-\gamma_-) - \matrice{0 & \ove{\var_\star}(e^{-it}-1) \\ \var_\star(e^{it}-1) & 0}\right)\matrice{a \\ b} = 0.
\end{equation}
In particular the matrix involved in \eqref{eq:2h} is singular and its determinant vanishes. It implies $(\gamma_+-\gamma_-)^2 = - \var_F^2 |e^{it}-1|^2$. Since $\Im \gamma_+ > 0$ and $\Im \gamma_- < 0$ we get $\gamma_+-\gamma_- = i\var_F \cdot |e^{it}-1|$. This relation allows to solve \eqref{eq:2h} for $[a,b]^\top$: 
\begin{equation}\label{eq:2j}
\matrice{a \\ b} \in \C \cdot \matrice{i\var_F \cdot |e^{it}-1| \\ -\var_\star(e^{it}-1)}.
\end{equation}
The second eigenvalue problem in \eqref{eq:2i} corresponds to the system
\begin{equation}
\systeme{\gamma_- a - \ove{\var_\star} b - \lambda a = 0 \\ -\gamma_- b - \var_\star a - \lambda b = 0 }
\end{equation}
and we deduce that $\ove{\var_\star} b^2 + \var_\star a^2 + 2\lambda ab = 0$. Because of \eqref{eq:2j}, this yields
\begin{equation}\label{eq:3o}
\lambda = -\dfrac{\ove{\var_\star}b^2+\var_\star a^2}{2ab} 
% = \dfrac{\var_\star \var_F^2(e^{it}-1)^2- \var_\star \var_F^2 \cdot |e^{it}-1|^2}{2 i\var_F \var_\star \cdot |e^{it}-1|(e^{it}-1) } 
= \dfrac{\var_F\sin(t)}{|e^{it}-1|} = \var_F\cos(t/2).
\end{equation}
\indent So far we showed that if $\lambda$ is an eigenvalue of $\Di_1(t)$ then $\lambda = \var_F \cos(t/2)$. To prove that the converse holds, we must check that there exist $\gamma_\pm, a, b$ solving \eqref{eq:2i} with $\lambda = \var_F \cos(t/2)$. The eigenvalue problems \eqref{eq:2i} have solutions 
\begin{equation}\label{eq:5s}
\var_F \cos(t/2) = \lambda = \pm \sqrt{\gamma_+^2+\var_F^2} = \pm \sqrt{\gamma_-^2+\var_F^2}.
\end{equation}
When $t \in (0,2\pi)$, $\cos(t/2)$ belongs to $(-1,1)$. Thus solutions $\gamma_\pm$ of \eqref{eq:5s} have a non-zero imaginary part; and if $\gamma_\pm$ is solution then so is $-\gamma_\pm$. Therefore there is a unique solution of \eqref{eq:5s} with $\Im \gamma_\pm > 0$. This proves the converse part and concludes the proof of the lemma when $\nu_\star > 0$.

If $\nu_\star < 0$, then the solution to \eqref{eq:2j} becomes
\begin{equation}
\matrice{a \\ b} \in \C \matrice{i\var_F \cdot |e^{it}-1| \\ \var_\star(e^{it}-1)}.
\end{equation}
In other words, $b$ changes sign while $a$ remains unchanged. It follows from \eqref{eq:3o} that $\lambda = -\var_F\cos(t/2) = \sgn(\nu_\star) \var_F \cdot \cos(t/2)$.
 \end{proof}

Hence when $\chi_+= \1_{[0,\infty)}$ and $\chi_-=\1_{(-\infty,0]}$ we can compute precisely the spectrum of $\Di_1(t)$. The corresponding spectral flow equals $\sgn(\nu_\star)$: a single eigenvalue flows either upwards or downwards as $t$ runs from $0$ to $2\pi$, depending on the sign of $\nu_\star$.\footnote{The identity $\Sf(\Di_1) = \sgn(\nu_\star)$ may possibly be contained in more general results \cite{FSF,Ba2}. The proof via Lemma \ref{lem:5d} is nonetheless very simple.} See Figure \ref{fig:4}. Because the spectral flow is a topological invariant, it remains equal to $\sgn(\nu_\star)$ even when $\chi_\pm$ are not piecewise constant -- see the argument of \S\ref{sec:3.3a} and \S\ref{sec:3.5b}. Together with \eqref{eq:2k} and \eqref{eq:1b}, we deduce Theorem \ref{thm:1}: 
\begin{equation}
\Sf_{\Tt^1}\big(\PP-E(1,\cdot)\big) = \sgn(\nu_\star) \cdot m = (-1)^{\frac{n-1}{2}} \cdot m,
\end{equation}
where we used $\sgn(\nu_\star) = (-1)^{\frac{n-1}{2}}$ -- see \eqref{eq:9t}.

\begin{center}
\begin{figure}
{\caption{On the left, the spectrum of $\Di_1(t)$ when $\chi_+= \1_{[0,\infty)}$, $\chi_-=\1_{(-\infty,0]}$ and $\nu_\star > 0$. The gray region represents the essential spectrum and the red curve is the eigenvalue. It flows from the top to the bottom of the spectrum as $t$ runs from $0$ to $2\pi$. The spectral flow equals $1$. On the right, a representation of the spectrum of $\DDi$ when $m=1$. It may look quite different from that of $\Di_1$ but the intersection number of the red curve with the horizontal axis must be the same.}\label{fig:4}}
{\includegraphics{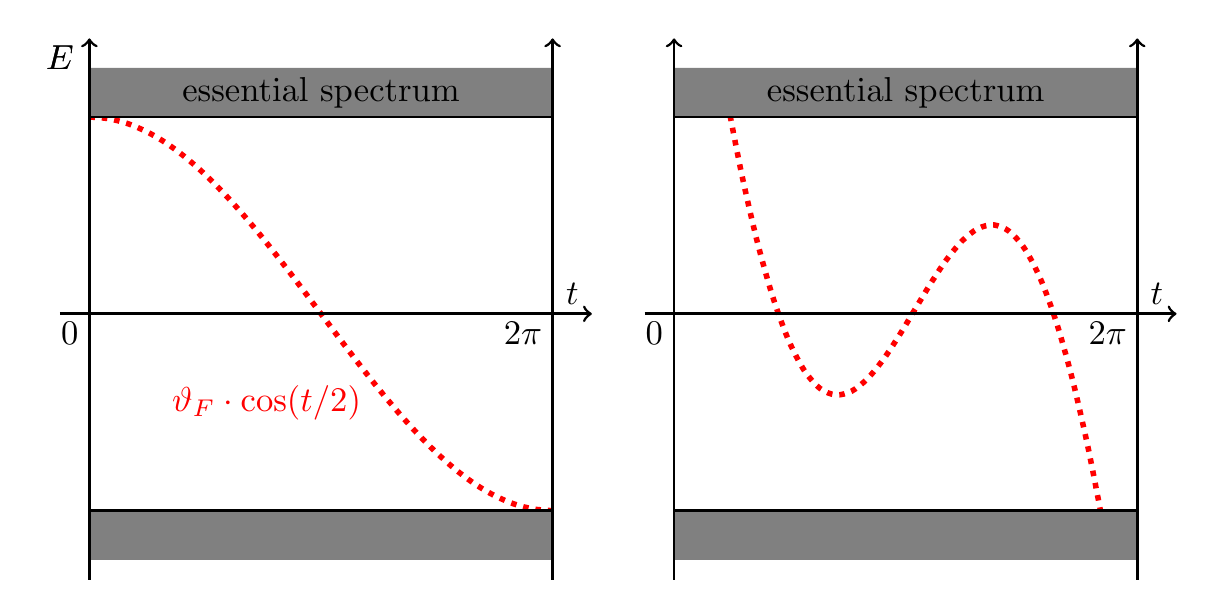}}
\end{figure}
\end{center}
\vspace{-1cm}

\section{Chern character and Chern class of some eigenbundles}\label{sec:4}

In the last section we expressed the edge index as a winding number involving $W$ and $\phi_\pm^\star$. We go on to the computation of the bulk index. In this section we collect some well-known results about  Berry connection, Berry curvature \cite{Si, Be} and (first) Chern class on rank-one eigenbundles with two-dimensional toric bases, without proofs. For an introduction to vectors bundles, connections, curvatures forms and characteristic classes, we recommend \cite[\S1]{BGV}.

\subsection{Unitary bundles over the two-torus}\label{sec:1z.1} Let $\Tt^2$ be the two-torus $\R^2/(2\pi\Z)^2$. We cover $\Tt^2$ with open sets $U_\az$, each diffeomorphic to open sets $V_\az \subset \R^2$. We will write $\zeta_\az \in V_\az$ for the coordinate of $\zeta \in U_\az$. Let $\EE \rightarrow \Tt^2$ be a smooth vector bundle of rank $n$, such that
\begin{equation}\label{eq:1zg}
\zeta = (\xi,t) \in \R^2/(2\pi \Z)^2 \ \Rightarrow \ \EE_\zeta \subset L^2_\xi.
\end{equation}
This bundle inherits the Hermitian metric induced by the Hermitian product on $L^2_\xi$.

The quotient map $\R^2 \rightarrow \Tt^2$ is surjective; therefore $\EE$ induces naturally a smooth Hermitian bundle $\Ee \rightarrow \R^2$. Since $\R^2$ is contractible, $\Ee$ is trivial. Let $\Psi = \{\Psi_j\}_{j=1}^n$ be a smooth unitary frame of $\Ee$: for each $z \in \R^2$, $\Psi(z) = \{\Psi(z)_j\}_{j=1}^n$ is a unitary basis of $\Ee_z$. If $\zeta \in U_\az \cap U_\beta \subset \Tt^2$, both $\{\Psi(\zeta_\az)_j\}_{j=1}^n$ and $\{\Psi(\zeta_\beta)_j\}_{j=1}^n$ are unitary basis of $\EE_\zeta$. Therefore, $\blr{\Psi\big(\zeta_\beta)_j,\Psi(\zeta_\az)_k}$ form a unitary matrix: for every $i, k$, 
\begin{equation}\label{eq:1zc}
 \delta_{ik} = \sum_{j=1}^n  \overline{\bblr{\Psi(\zeta_\beta)_j,\Psi(\zeta_\az)_i}} \cdot \bblr{\Psi(\zeta_\beta)_j,\Psi(\zeta_\az)_k} .
\end{equation}
We use $\Psi$ to define coordinates on $\EE$: a point $(\zeta,f) \in \EE$ with $\zeta \in U_\az$ has coordinates $\big(\zeta_\az, \blr{\Psi(\zeta_\az),f } \big) \in \R^2 \times \C^n$. We remark that if $\zeta \in U_\az \cap U_\beta$, then $(\zeta,f)$ also has coordinates 
$\left(\zeta_\beta, \blr{\Psi(\zeta_\beta),f } \right)$. Therefore, the transition functions $t_{\az \beta}(\zeta) \in M_n(\C)$ associated to this system of coordinates satisfy
\begin{equation}
f \in \EE_\zeta \ \ \Rightarrow \ \ \bblr{\Psi(\zeta_\az)_i,f } = \sum_{j=1}^n t_{\az\beta}(\zeta)_{ij} \cdot \bblr{\Psi(\zeta_\beta)_j,f }.
\end{equation}
Taking $f = \Psi(\zeta_\az)_k$ and recalling that $\blr{\Psi(\zeta_\az)_i,\Psi(\zeta_\az)_k} = \delta_{ik}$, we deduce from \eqref{eq:1zc} an expression for the transition functions:
\begin{equations}
 t_{\az\beta}(\zeta)_{ij} = \overline{\bblr{\Psi(\zeta_\beta)_j,\Psi(\zeta_\az)_i}} = \bblr{\Psi(\zeta_\az)_i,\Psi(\zeta_\beta)_j}.
\end{equations}
In particular,
\begin{equation}\label{eq:1zf}
\Psi(\zeta_\az)_i = \sum_{j=1}^n \bblr{\Psi(\zeta_\beta)_j,\Psi(\zeta_\az)_i} \cdot \Psi(\zeta_\beta)_j = \sum_{j=1}^n \overline{t_{\az\beta}(\zeta)_{ij}} \cdot \Psi(\zeta_\beta)_j.
\end{equation}

\subsection{Sections} We define smooth sections of $\EE$ are functions $\sigma : M \rightarrow \cup_{\zeta \in \Tt^2} \EE_\zeta$ such that the map
\begin{equation}
\zeta \in \Tt^2 \mapsto \big( \zeta, \sigma(\zeta) \big) \in \EE
\end{equation}
is smooth. We denote the space of smooth section of $\EE$ by $\Gamma\big(\Tt^2,\EE\big)$. We define similarly sections of $\Ee$.

 We observe that a section of $\EE$ always induces a section of $\Ee$. On the other hand, if $s$ is a section of $\Ee$, then $s$ induces a section $\sigma$ of $\EE$ if and only if for every $\zeta \in U_\az \cap U_\beta \subset \Tt^2$,
\begin{equation}\label{eq:1zd}
\lr{\Psi(\zeta_\az)_i,s(\zeta_\az)} = \sum_{j=1}^n t_{\az\beta}(\zeta)_{ij} \cdot \bblr{\Psi(\zeta_\beta)_j ,s(\zeta_\beta)}.
\end{equation}
If we write $s = \sum_{j=1}^n s_j \cdot \Psi_j$, this is equivalent to the compatibility condition $s(\zeta_\az)_i = \sum_{j=1}^n t_{\az\beta}(\zeta)_{ij} \cdot s(\zeta_\beta)_j$.

\subsection{Connections on $\EE$} For $(\xi,t) \in \R^2$, we define $\Pi(\xi,t)$ as the orthogonal projector from $L^2_\xi$ to $\Ee_{\xi,t}$. Given $s$ a section of $\Ee$, we write $s = \sum_{k=1}^n s_k \cdot \Psi_k$, $ds_j = \p_ts_k \cdot dt + \p_\xi s_k \cdot d\xi$ and we define $\nabla s$ as:
\begin{equation}\label{eq:1za}
\nabla s \de \sum_{k=1}^n ds_k \cdot \Psi_k  + s_k \cdot \Pi \Big( \p_t \Psi_k \cdot dt + \nabla_\xi \Psi_k \cdot d\xi \Big), \ \ \ \ \ \nabla_\xi \Psi_k \de e^{i\xi x} \cdot \dd{\big(e^{-i\xi x} \Psi_k\big)}{\xi}.
\end{equation} 
This is a connection: a first order differential operator from $\Gamma\big(\R^2,\Ee\big)$ to $\Gamma\big(\R^2,\Ee \otimes T^*\R^2\big)$ that respects Leibnitz's rule.
In a more general framework, \eqref{eq:1za} defines the Grassmann connection on $\Ee$. It is the natural projection of the trivial connection on the bundle with fibers $L^2_\xi$ on $\Ee$, which is defined as
\begin{equation}
\nabla g = \dd{g}{t} \cdot dt + e^{i\xi x} \cdot \dd{\big(e^{-i\xi x} f\big)}{\xi} \cdot d\xi = \dd{g}{t} \cdot dt + \nabla_\xi g \cdot d\xi.
\end{equation}
for every $g : \R^2 \rightarrow \cup_\xi L^2_\xi$ such that for every $(\xi, t)$, $e^{-i\xi x} g(\xi,t) \in L^2_0$ and varies smoothly. In this framework it is known as the Berry connection.

\begin{lem}\label{lem:1zza} If $s \in \Gamma(\R^2,\Ee)$ induces a section of $\EE$, then $\nabla s$ -- initially defined as a section of $\Ee \otimes T^*\R^2$ -- induces a section of $\EE\otimes T^*\Tt^2$.
\end{lem}

Lemma \ref{lem:1zza} shows that $\nabla$ induces a connection on $\EE$.

\subsection{Curvature} We review the curvature of $\nabla$. Given $X, Y \in T\Tt^2$ and $\sigma \in \Gamma\big(\Tt^2,\EE\big)$, we set $\nabla_X \sigma = \nabla \sigma(X) \in \Gamma(\Tt^2,\EE)$ and 
\begin{equation}
\BB(X,Y) = \nabla_X\nabla_Y - \nabla_Y \nabla_X : \Gamma(\Tt^2,\EE) \rightarrow \Gamma(\Tt^2,\EE).
\end{equation}
The operator $(\sigma,X,Y) \mapsto \BB(X,Y) \sigma$ is a tensor field: its evaluation at $\zeta$ depends only on $\sigma(\zeta)$, $X(\zeta)$ and $Y(\zeta)$. Moreover, $\BB(X,Y) \sigma = - \BB(Y,X) \sigma$. Therefore, this operator is a $\End(\EE)$-valued two-form on $\Tt^2$, given by $\BB(\p_\xi,\p_t) \cdot d\xi \wedge dt$. Since $\p_\xi$ and $\p_t$ are global vector fields on $\Tt^2$, $\BB(\p_\xi,\p_t)$ is a smooth section of $\End(\EE)$. 

\begin{lem}\label{lem:1zzb} The trace of $\BB(\p_t,\p_\xi)$ is a smooth function on $\EE$, given in local coordinates by:
\begin{equations}\label{eq:1zb}
B(\zeta) = 2i \cdot \Im \sum_{j=1}^n \bblr{ \nabla_\xi\Psi(\zeta_\az)_j,\nabla_t \Psi(\zeta_\az)_j } 
= \Tr\Big( \Pi(\zeta_\az) \big[ \nabla_\xi \Pi(\zeta_\az), \p_t  \Pi(\zeta_\az) \big] \Big), \\ \nabla_\xi \Pi(\zeta_\az) \de e^{i\xi x} \cdot \dd{ e^{-i\xi x} \Pi(\zeta_\az) e^{i\xi x}}{\xi} \cdot e^{-i\xi x}.
\end{equations}
\end{lem}

So far, everything depended on the frame $\Psi$ of \S\ref{sec:1z.1} for the bundle $\Ee$. If $\Phi$ was another frame, then $\Psi$ and $\Phi$ would be related by a family of unitary maps $\Omega : \R^2 \rightarrow U(n)$, via:
$\Psi(z)_i = \Omega(z)_{ij} \Psi)(z)_j$. However, the formula \eqref{eq:1zb} is manifestly gauge-invariant: it depends only on the projector $\Pi(z)$, which does not depend on the choice of frame. In other words, the function $B \in C^\infty\big(\Tt^2,\C\big)$ depends only on $\EE$. It is called the Berry curvature of $\EE$.

\subsection{Chern number}\label{sec:8.3} We define the first Chern class of $\EE$ as:
\begin{equation}\label{eq:1zh}
c_1(\EE) \de \dfrac{1}{2\pi i} \int_{\Tt^2} B(\zeta) d\zeta.
\end{equation}

\begin{lem}\label{lem:1zzc} The Chern class of $\EE$ is an integer.
\end{lem}

This is well-known is analog to the Gauss Bonnet theorem. The bundle $\EE$ is fully characterized by the family of projectors $\Pi$. Let $\EE_\delta \rightarrow \Tt^2$ be a family of bundles satisfying \eqref{eq:1zg} and  depending smoothly on $\delta \in (0,\delta_0)$ -- in the sense that the bundle with fiber $\EE_{\delta,\zeta}$ over $\Tt^2 \times (0,\delta_0)$ is smooth. Then: the associated projectors $\Pi_\delta(\zeta)$; the traces of Berry curvatures $B_\delta(\zeta)$; and the Chern classes $c_1(\EE_\delta)$ depend smoothly on $(\delta,\zeta)$. Since $c_1(\EE_\delta) $ lives in a discrete set, it is constant. In other words, $c_1(\EE)$ is a topological invariant of $\EE$.

 To compute $c_1(\EE)$, we can pick any unitary frame $\Psi$ of $\Ee$ and apply \eqref{eq:1zh} combined with \eqref{eq:1zb}. This yields:
 \begin{equation}
 c_1(\EE) = \dfrac{1}{2\pi i} \int_{[0,2\pi]^2} \Tr\Big( \Pi \big[ \nabla_\xi \Pi, \p_t \Pi \big] \Big)(\xi,t) \cdot d\xi dt.
 \end{equation}

\subsection{Bulk invariant}\label{sec:8.4}

We define here the bulk invariant associated to the family $P(\xi,t) = D_x^2 + V + W_t  :  L^2_\xi \rightarrow L^2_\xi$, where $V$, $W$ satisfy \eqref{eq:9v} and  $\oA$, $\oB$. We denote by $\lambda_{s,1}(\xi,t) \leq \dots \leq \lambda_{s,\ell}(\xi,t) \leq \dots$ its $L^2_\xi$-eigenvalues.  Let $\Pi_s(\xi,t) : L^2_\xi \rightarrow L^2_\xi$ be the orthogonal projector  to
\begin{equation}\label{eq:1a}
\bigoplus_{j=1}^n \ker\big(P_s(\xi,t)-\lambda_{s,j}(\xi,t)\big).
\end{equation}
Because of $\oA$, the $n$-th eigenvalue of $P_1(\pi,t)$ is simple. Since $n$ is odd and using the monotonicity properties of the dispersion curves, we have
\begin{equation}\label{eq:1e}
s\in (0,1] \ \Rightarrow \ 
\lambda_{s,n}(\xi,t) \leq \lambda_{s,n}(\pi,t) < \lambda_{s,n+1}(\pi,t) \leq \lambda_{s,n+1}(\xi,t).
\end{equation}
In particular, the first $n$-th eigenvalues of $P_s(\xi,t)$ are separated from all other eigenvalues of $P_s(\xi,t)$. Hence, $\Pi_s(\xi,t)$
 varies smoothly with $(\xi,t)$ -- in the sense that $e^{-i\xi x} \Pi_s(\xi,t) e^{i\xi x}$ is a smooth family of operators on $L^2_0$ -- see \cite[\S VII.1.3, Theorem 1.7]{Ka}. Because $P_s(\xi,t)$ depends periodically on $(\xi,t)$, we can define a smooth rank-$n$ bundle $\EE_s$ over $\Tt^2 = \R^2/(2\pi\Z)^2$ whose fiber at $(\xi,t)$ is the space \eqref{eq:1a}. The bulk invariant of $P_s$ is $c_1(\EE_s)$.

Because of \cite[\S VII.1.3, Theorem 1.7]{Ka} and of \eqref{eq:1e}, the projector $\Pi_s(\xi,t)$ depends analytically on $s \in (0,1], \xi$ and $t$. Therefore $c_1(\EE_s)$ is a continuous function of $s$. It follows that $c_1(\EE_s)$ is constant. We deduce that
\begin{equation}
c_1(\EE_1) = \lim_{\delta \rightarrow 0} c_1(\EE_\delta).
\end{equation}
In other words, it suffices to compute $c_1(\EE_\delta)$ in the limit $\delta \rightarrow 0$ in order to prove Theorem \ref{thm:2}.

\section{Computation of the bulk index}\label{sec:5}  

In this section, we compute the bulk index. We fix $n$ an odd integer and $(V,W)$ a pair of periodic potentials satisfying \eqref{eq:9v} together with $\oA$ and $\oB$. Under these conditions, the $n$-th $L^2_\pi$-eigenvalue of $D_x^2+V$ corresponds to a Dirac point $(\pi,E_\star)$ -- see \S\ref{sec:2.4}; and the bundle $\EE$ of \S\ref{sec:8.4} is correctly defined. Theorem \ref{thm:2} expresses the first Chern class of the bundle $\EE$ in terms of $W$ and the Dirac eigenbasis $(\phi_+^\star,\phi_-^\star)$ associated to $(\pi,E_\star)$:
\begin{equation}\label{eq:1c}
c_1(\EE) = \dfrac{(-1)^{\frac{n-1}{2}}}{2\pi i} \int_0^{2\pi} \dfrac{\var'(t)}{\var(t)} dt, \ \ \ \ \  \ \var(t) \de \blr{\phi_-^\star,W_t \phi_+^\star}.
\end{equation}

We briefly explain the proof. 
Because of topological invariance of Chern classes, deforming $P(\xi,t)$ to $P_{\delta}(\xi,t) = D_x^2 + V + \delta W_t$ preserves the Chern class: $c_1(\EE) = c_1(\EE_\delta)$. In the asymptotic regime $\delta \rightarrow 0$, we derive estimates for the Berry curvature of $\EE_\delta$,
\begin{equation}
B_\delta(\xi,t) \de \Tr\left( \Pi_\delta(\xi,t) \big[ \nabla_\xi\Pi_\delta(\xi,t), \p_t  \Pi_\delta(\xi,t) \big] \right), \ \ \ \ (\xi,t) \in [0,2\pi]^2.
\end{equation}
Above, $\Pi_\delta(\xi,t)$ is the $L^2_\xi$-projector to the space spanned by the $n$ first $L^2_\xi$-eigenvectors of $P_\delta(\xi,t)$. We show:
\begin{itemize}
\item  $B_\delta(\xi,t)$ is negligible when $\xi$ is away from $\pi$ -- Lemma \ref{lem:1f};
\item For $\xi$ near $\pi$ and $\delta$ small, $P_{\delta}(\xi,t)$ behaves like a $2 \times 2$  tight-binding model with Hamiltonian $\Mm_\delta(\xi,t)$  -- Lemma \ref{lem:1b}; 
\item After rescaling, $B_\delta(\xi,t)$ approaches the Berry curvature of the low-energy eigenbundle associated to $\Mm_\delta(\xi,t)$ for $\xi$ near $\pi$ -- Lemma \ref{lem:2b}. 
\item The family $\Mm_\delta(\xi,t)$ is linear in $\xi$ and depends on $t$ only through $\var(t)$. Its low-energy Chern number is well-defined as the integral of the corresponding Berry curvature. It equals the winding number appearing in \eqref{eq:1c} -- Lemma \ref{lem:2c}.
\end{itemize}
These statements are conform with the physical intuition: $B_\delta(\xi,t)$ is greatest near degeneracies; and near such singularities, a tight-binding model captures the dominant effects. Along the proof, we will need some simple estimates about the spectrum of $P_{\delta}(\xi,t)$. These are stated and proved in Appendix \ref{sec:7.2}.

\begin{center}
\begin{figure}
{\caption{The operator $D_x^2+V$ has a Dirac point at $(\pi,E_\star)$: $E_\star$ is an $L^2_\pi$-eigenvalue of $D_x^2+V$ of multiplicity $2$. The corresponding dispersion curves are not smooth: see the left graph. Adding the potential $\delta W_t$ opens a small gap near the energy $E_\star$; the corresponding dispersion curves are smooth but their second order derivatives blow up at $\xi = \pi$ as $\delta$ goes to $0$: see the right graph.}\label{fig:5}}
{\includegraphics{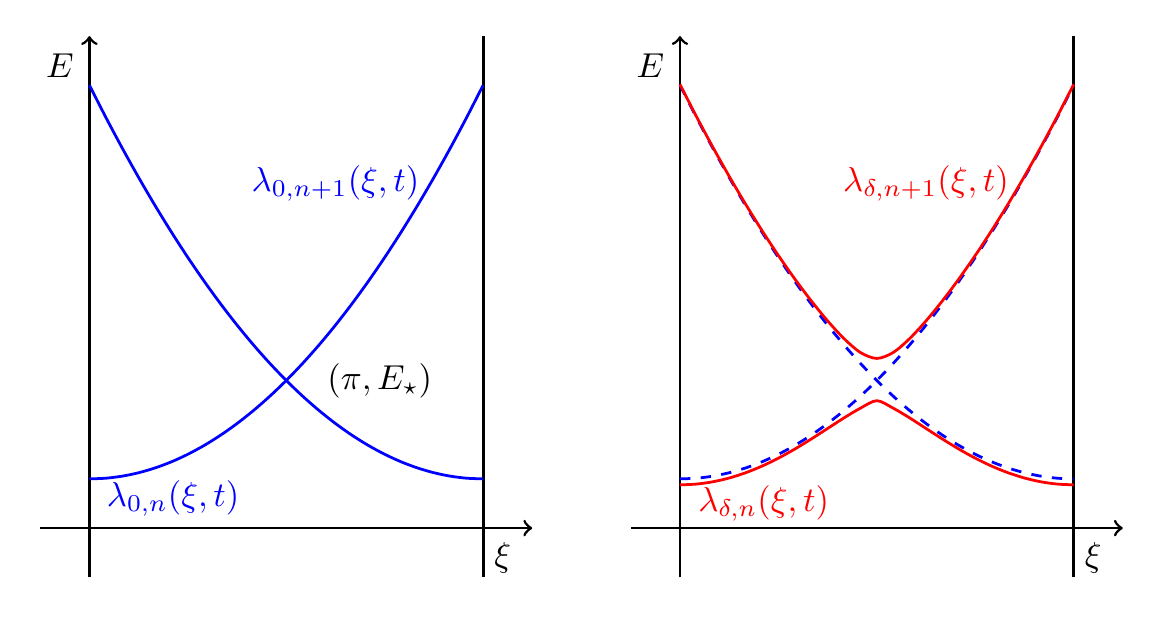}}
\end{figure}
\end{center}
\vspace*{-.9cm}

\subsection{Reduction to $\xi$ near $\pi$} 

\begin{lem}\label{lem:1d} For every $\epsi > 0$, there exists $C > 0$ such that
\begin{equation}\label{eq:1f}
t \in \R, \ \ \xi \in [0,\pi - \epsi] \cup [\pi+\epsi,2\pi], \ \delta \in (0,1] \ \Rightarrow \ B_\delta(\xi,t) \leq C\delta.
\end{equation}
\end{lem}
 
\begin{proof} 1. We observe that $\Pi_0(\xi,t)$ is well defined when $\xi \notin \pi+2\pi\Z$: the eigenvalue $\lambda_{0,n}(\xi,t)$ is degenerated only at $\xi=\pi$. Therefore, the family of projectors
\begin{equation}
(\delta,\xi,t) \in [0,1] \times [0,2\pi] \setminus \{\pi\} \times \R \mapsto \Pi_\delta(\xi,t) 
\end{equation}
is smooth -- \cite[\S VII.1.3, Theorem 1.7]{Ka}. If $\epsi > 0$ is given then $[0,\pi - \epsi] \cup [\pi+\epsi,2\pi]$ is a compact subset of $[0,2\pi] \setminus \{\pi\}$. It follows that $\Pi_\delta(\xi,t) - \Pi_0(\xi,t) = \OO_{L^2_\xi}(\delta)$, together with its derivatives in $\xi$ and $t$. Moreover, $\Pi_0(\xi,t)$ does not depend on $t$. We deduce that in the range specified in \eqref{eq:1f},
\begin{equation}
\Pi_\delta(\xi,t) \big[ \nabla_\xi \Pi_\delta(\xi,t), \p_t \Pi_\delta(\xi,t) \big] = \OO_{L^2_\xi}(\delta).
\end{equation}
The operator in the LHS has rank bounded by $n$. Therefore we can the trace on both side while preserving the order of the order term. This yields 
\begin{equation}
\Tr\left(\Pi_\delta(\xi,t) \big[ \nabla_\xi \Pi_\delta(\xi,t), \p_t \Pi_\delta(\xi,t) \big]\right) = O(\delta),
\end{equation}
which completes the proof.
\end{proof}

\subsection{$\xi$ near $\pi$.} We now estimate the Berry curvature when $|\xi-\pi| \leq \delta^{3/4}$. Let $\{\phi_+(\xi), \phi_-(\xi)\}$ be the smoothly varying basis of eigenvectors of $P_{0}(\xi,t)$ near the Dirac point; see \S\ref{sec:2.4}. Define $J_0(\xi) : L^2_\xi \rightarrow \C^2$ by
\begin{equation}\label{eq:2l}
J_0(\xi) v \de \matrice{ \blr{\phi_+(\xi),v} \\  \blr{\phi_-(\xi),v}}.
\end{equation}
We introduce the $2 \times 2$ matrix
\begin{equation}
\Mm_\delta(\xi,t) \de \matrice{E_\star + \nu_\star(\xi-\pi) & \delta \ove{\var(t)} \\ \delta \var(t) & E_\star-\nu_\star(\xi-\pi)}.
\end{equation}
It has two distinct eigenvalues, 
\begin{equation}
\mu^{\pm}_\delta(\xi,t) \de E_\star \pm r_\delta(\xi,t), \ \ \ \ r_\delta(\xi,t) \de \sqrt{\nu_\star^2(\xi-\pi)^2 + \delta^2 |\var(t)|^2}.
\end{equation}
The next lemma shows that in a certain energy regime, $P_{\delta}(\xi,t)$ behaves very much like $\Mm_\delta(\xi,t)$. A similar result was used in the proof of \cite[Theorem 4]{DFW}: see \cite[Proposition 4.2]{DFW}. 

\begin{lem}\label{lem:1b} There exists $\delta_0 > 0$ and $\epsi_0 \in (0,1)$ such that
\begin{equations}
\delta \in (0,\delta_0), \ \ |\xi-\pi| \leq \epsi_0, \ \ t \in \R \ \ 
z \in \p\BBb\left(\mu_\delta^\pm(\xi,t),r_\delta(\xi,t)\right)
\\
\Rightarrow \
\big( P_{\delta}(\xi,t)-z \big)^{-1} = J_0(\xi)^* \cdot \big(\Mm_\delta(\xi,t)-z \big)^{-1}\cdot J_0(\xi) + \OO_{L^2_\xi}(1).
\end{equations}
\end{lem}

\begin{proof} 1. We use $\zeta = (\xi,t)$. We prove the lemma for $\mu_\delta^-(\zeta)$ only -- the proof for $\mu_\delta^+(\zeta)$ is identical. Let $\VVV(\xi) \subset L^2_\xi$ be the two-dimensional vector space
\begin{equation}\label{eq:6d}
\VVV(\xi) \de \C \phi_+(\xi) \oplus \C \phi_-(\xi) \subset L^2_\xi.
\end{equation}
We write $P_{\delta}(\zeta)-z$ as a block matrix with respect to the splitting $L^2_\xi = \VVV(\xi) \oplus \VVV(\xi)^\perp$:
\begin{equation}
P_{\delta}(\zeta) - z \de \matrice{A_\delta(\zeta,z) & B_\delta(\zeta) \\ C_\delta(\zeta) & D_\delta(\zeta,z)}.
\end{equation}
Above, $A_\delta(\zeta,z) : \VVV(\xi) \rightarrow \VVV(\xi)$; $B_\delta(\zeta) : \VVV(\xi)^\perp \rightarrow \VVV(\xi)$; $C_\delta(\zeta) = B_\delta(\zeta)^* : \VVV(\xi) \rightarrow \VVV(\xi)^\perp$; and $D_\delta(\zeta,z) : \VVV(\xi)^\perp \rightarrow \VVV(\xi)^\perp$.
Schur's complement formula computes inverses of block by block operators under suitable conditions -- see \cite[Lemma 4.1]{DFW}. Here these conditions are:
\begin{equations}
A_\delta(\zeta,z) \ : \ \VVV(\xi) \rightarrow \VVV(\xi) \ \text{ is invertible and } 
\\
D_\delta(\zeta,z) - C_\delta(\zeta)A_\delta(\zeta,z)^{-1}B_\delta(\zeta) \ : \ \VVV(\xi)^\perp \rightarrow \VVV(\xi)^\perp \ \text{ is invertible.}
\end{equations}
We show that they hold in Steps 2 and 3 below. 

2. Let $I_0(\xi) : \VVV(\xi) \rightarrow \C^2$ be the coordinate map in the basis $\{\phi_+(\xi),\phi_-(\xi)\}$. The operator $I_0(\xi)$ has the same expression \eqref{eq:2l} as $J_0(\xi)$ -- though $v \in \VVV(\xi)$ instead of $L^2_\xi$. Moreover $I_0(\xi)$ is invertible, and
\begin{equation}\label{eq:5f}
A_\delta(\zeta,z) = I_0(\xi)^{-1} \matrice{\lambda_+(\xi)-z & \delta \blr{\phi_+(\xi),W_t\phi_-(\xi)} \\  \delta \blr{\phi_-(\xi),W_t\phi_+(\xi)} & \lambda_-(\xi)-z } I_0(\xi).
\end{equation}
To prove \eqref{eq:5f}, we observe that $\blr{\phi_\pm(\xi),W_t \phi_\pm(\xi)} = 0$ from Lemma \ref{lem:2z}.

We show that $A_\delta(\zeta,z)$ is invertible. We recall that $\lambda_{\pm}(\xi) = E_\star \pm \nu_\star(\xi-\pi) + O(\xi-\pi)^2$ and we observe that
\begin{equation}
\delta \blr{\phi_-(\xi),W_t\phi_+(\xi)} = \delta \blr{\phi_-^\star,W_t\phi_+^\star} + O\big(\delta(\xi-\pi)\big) = \delta \var(t) + O\left(r_\delta(\zeta)^2\right).
\end{equation}
Hence
\begin{equations}\label{eq:1i}
A_\delta(\zeta,z) = I_0(\xi)^{-1} \big( \Mm_\delta(\zeta)-z\big) I_0(\xi) + \OO_{\VVV(\xi)}\left(r_\delta(\zeta)^2\right).
\end{equations}
The matrix $\Mm_\delta(\zeta)$ is selfadjoint. Therefore, a bound on $(\Mm_\delta(\zeta)-z)^{-1}$ follows from an estimate on the distance between $z$ and the spectrum of $\Mm_\delta(\zeta)$.
Observe that $|\mu_\delta^+(\zeta)-\mu_\delta^-(\zeta)| = 2 r_\delta(\zeta)$; therefore,
\begin{equation}
|z-\mu_\delta^-(\zeta)| = r_\delta(\zeta) \ \Rightarrow \ |z-\mu_\delta^+(\zeta)| \geq r_\delta(\zeta).
\end{equation}
It follows that,
\begin{equation}\label{eq:1j}
z \in \p\BBb\big(\mu_\delta^-(\zeta),r_\delta(\zeta)\big) \ \Rightarrow \ \left\|(\Mm_\delta(\zeta)-z)^{-1}\right\|_{\C^2} = O\big( r_\delta(\zeta)^{-1} \big). 
\end{equation}
Because of a perturbation argument based on \eqref{eq:1i}, \eqref{eq:1j} and a Neumann series, when $z \in \p\BBb(\mu_\delta^-(\zeta),r_\delta(\zeta))$ and $|\xi-\pi| + \delta$ is sufficiently small, $A_\delta(\zeta,z)$ is invertible. This comes with the estimates
\begin{equations}
A_\delta(\zeta,z)^{-1} = I_0(\xi)^{-1} \big( \Mm_\delta(\zeta)-z \big)^{-1} I_0(\xi) + \OO_{\VVV(\xi)}(1), \\ 
\left\| A_\delta(\zeta,z)^{-1} \right\|_{\VVV(\xi)} = O\left(r_\delta(\zeta)^{-1} \right),
\end{equations}
valid for $|\xi-\pi|$ and $\delta$ small enough.

3. The operator $B_\delta(\zeta) : \VVV(\xi)^\perp \rightarrow \VVV(\xi)$ is the composition of the embedding $\VVV(\xi)^\perp \rightarrow L^2_\xi$, the multiplication $\delta W_t$ and the projection $L^2_\xi \rightarrow \VVV(\xi)$. Therefore $B_\delta(\zeta) = \OO_{\VVV(\xi)^\perp \rightarrow \VVV(\xi)}(\delta)$. Since $C_\delta(\zeta) = B_\delta(\zeta)^*$, $C_\delta(\zeta) = \OO_{\VVV(\xi) \rightarrow \VVV(\xi)^\perp}(\delta)$. Step 2 implies
\begin{equations}
D_\delta(\zeta,z) - C_\delta(\zeta)A_\delta(\zeta,z)^{-1}B_\delta(\zeta)  = D_\delta(\zeta,z) + \OO_{\VVV(\xi)^\perp}(\delta)
\\
= (P_{0}(\zeta)-z)\big|_{\VVV(\xi)^\perp} + \OO_{\VVV(\xi)^\perp}(\delta).
\end{equations}
Above $(P_{0}(\zeta)-z)|_{\VVV(\xi)^\perp}$ is seen as an operator $\VVV(\xi)^\perp \rightarrow \VVV(\xi)^\perp$. It is invertible with inverse $\OO_{\VVV(\xi)^\perp}(1)$ when $z$ is sufficiently far from $E_\star$. This is because all eigenvalues of $P_{0}(\zeta)$ but the $n$-th and $n+1$-th must be at distance of order $1$ from $E_\star$ -- see Lemma \ref{lem:1h}. Once again, a Neumann series argument shows that $D_\delta(\zeta,z) - C_\delta(\zeta)A_\delta(\zeta,z)^{-1}B_\delta(\zeta)$ is also invertible with inverse $\OO_{\VVV(\xi)^\perp}(1)$.

4. We conclude that the inverse of $P_{\delta}(\zeta)-z$ exists and is given by Schur's complement formula:
\begin{equation}
\big( P_{\delta}(\zeta)-z \big)^{-1} = \matrice{A_\delta(\zeta,z)^{-1} & 0 \\ 0 & 0}  + \OO_{L^2_\xi}(1).
\end{equation}
The lemma follows.
\end{proof}

Introduce $\MM(\xi,t)$ the $2 \times 2$ matrix 
\begin{equation}\label{eq:2u}
\MM(\xi,t) \de \dfrac{1}{\delta} \left(\Mm_\delta\left(\pi+\dfrac{\delta \xi}{\nu_\star},t\right)-E_\star\right) = \matrice{\xi & \ove{\var(t)} \\ \var(t) & -\xi}.
\end{equation}
It has eigenvalues $\pm \sqrt{\xi^2+|\var(t)|^2}$. Let $\LL \rightarrow \R \times \Tt^1$ be the line bundle whose fiber at $(\xi,t)$ is the eigenspace 
\begin{equation}\label{eq:1y}
\ker_{\C^2}\big(\Mm_\delta(\xi,t) + \sqrt{\xi^2+|\var(t)|^2}\big)
\end{equation}
Following \S\ref{sec:4}, we define the Berry curvature $(\xi,t) \in \R \times \Tt^1 \mapsto \BB(\xi,t)$ of $\LL$. Since $\R \times \Tt^1$ is a non-compact manifold, it is not immediate that we can integrate $\BB$. 

\begin{lem}\label{lem:2c} The orthogonal projector $\pi^-(\xi,t) : \C^2 \rightarrow \C^2$ to the space \eqref{eq:1y} satisfies
\begin{equation}\label{eq:1u}
\sup_{\R \times [0,2\pi]} \Big\{ \big\| \p_\xi \pi^-(\xi,t) \big\|_{\C^2} + \big\| \p_t \pi^-(\xi,t) \big\|_{\C^2} \Big\} < \infty.
\end{equation}
Moreover, $\BB$ is integrable on $\R \times [0,2\pi]$ and
\begin{equation}
\dfrac{1}{2\pi i}\int_{\R \times [0,2\pi]} \BB(\xi,t) \cdot d\xi dt = \dfrac{1}{2\pi i} \int_0^{2\pi} \dfrac{\var'(t)}{\var(t)} dt.
\end{equation}
\end{lem}

The proof of this lemma is a calculation postponed to Appendix \ref{sec:7.4}. 
Because of \eqref{eq:2u} and \eqref{eq:1u}, we deduce that the orthogonal projector $\pi_\delta^-(\xi,t) : \C^2 \rightarrow \C^2$ to the negative eigenspace of $\Mm_\delta(\xi,t)$ satisfies
\begin{equation}\label{eq:1v}
\sup_{\R \times [0,2\pi]} \big\| \p_\xi \pi_\delta^-(\xi,t) \big\|_{\C^2} = O\big(\delta^{-1}\big), \ \ \ \ \sup_{\R \times [0,2\pi]} \big\| \p_t \pi_\delta^-(\xi,t) \big\|_{\C^2} < \infty.
\end{equation}
For $(\xi,t) \in \R^2$, we let $\Bb_\delta(\xi,t)$ be the Berry curvature associated to the low-lying eigenbundle of $\Mm_\delta(\xi,t)$, i.e. associated to $\pi_\delta^-(\xi,t)$.

\begin{lem}\label{lem:2b} Let $\epsi_0$ be specified by Lemma \ref{lem:1b}.
There exists $\delta_0 > 0$ such that 
\begin{equations}\label{eq:6f}
\delta \in (0,\delta_0), \ \ |\xi-\pi| \leq \epsi_0, \ \ t \in \R 
 \ \ \ 
\Rightarrow \ \ \ 
B_\delta(\xi,t) = \Bb_\delta(\xi,t) + O(1).
\end{equations}
\end{lem}

\begin{proof} 1. We write $\zeta = (\xi,t)$. Let $Q_\delta(\zeta) : L^2_\xi \rightarrow L^2_\xi$ the projector on 
\begin{equation}
\bigoplus_{j=1}^{n-1} \ker_{L^2_\xi}\big(P_\delta(\zeta)-\lambda_{\delta,j}(\xi) \big).
\end{equation}
When $|\xi-\pi| \leq \epsi_0 < 1$, the eigenvalues $\lambda_{0,1}(\zeta), \dots, \lambda_{0,n-1}(\zeta)$ of $P_0(\zeta)$ are separated from the rest of the spectrum of $P_0(\zeta)$. Indeed, $\lambda_{0,n-1}(\zeta)$ can equal $\lambda_{0,n}(\zeta)$ only if $\xi = 0$, which is excluded because $|\xi-\pi| < 1$. Because of \cite[\S VIII.1.3 Theorem 1.7]{Ka}, the family $(\delta,\zeta) \mapsto Q_\delta(\zeta)$ is smooth on a neighborhood of $\{(0,\pi)\} \times \R$. In particular, under the conditions of \eqref{eq:6f}, 
\begin{equation}
\Tr\Big( Q_\delta(\zeta) \big[ \nabla_\xi Q_\delta(\zeta), \p_t Q_\delta(\zeta)\big] \Big) = O(1).
\end{equation}

For $\delta > 0$, let $\Pi_\delta^-(\xi)$ be the projector on $\ker_{L^2_\xi}\big(\Pp_{\delta,+}(\xi)-\lambda_{\delta,n}(\xi) \big)$. For $(\delta,\xi)$ near $(0,\xi_\star)$, $\lambda_{\delta,n}(\xi)$ is a simple eigenvalue of $P_\delta(\xi)$ -- see Lemma \ref{lem:1h}. Therefore for $(\delta,\xi)$ near $(0,\xi_\star)$ the projector $\Pi_\delta(\xi)$ splits orthogonally as
\begin{equation}\label{eq:1w}
\Pi_\delta(\xi) = \Pi_\delta^-(\xi) + Q_\delta(\xi).
\end{equation}

We recall that $B_\delta(\zeta)$ is gauge independent -- i.e. it does not depend on the choice of frame of $\EE_\delta$ used in \S\ref{sec:1z.1}. Pick a frame of $\EE_\delta$ associated to the orthogonal decomposition \eqref{eq:1w}. The associated Berry connection and curvature split accordingly to components for $\Pi_\delta^-(\zeta)$ and $Q_\delta(\zeta)$, see Lemma \ref{lem:1zzb}. Thanks to \eqref{eq:1w}, we deduce that
\begin{equations}\label{eq:1o}
B_\delta(\zeta) = \Tr\Big( Q_\delta(\zeta) \big[ \nabla_\xi Q_\delta(\zeta),  \p_t Q_\delta(\zeta)\big] \Big) + \Tr\Big( \Pi_\delta^-(\zeta) \big[ \nabla_\xi\Pi_\delta^-(\zeta), \p_t \Pi_\delta^-(\zeta)\big] \Big) 
\\
= \Tr\Big( \Pi_\delta^-(\zeta) \big[ \nabla_\xi \Pi_\delta^-(\zeta), \p_t \Pi_\delta^-(\zeta)\big] \Big) + O(1).
\end{equations}

2. Let $\gamma_\delta(\zeta) = \p\Bb\big(\mu_\delta^-(\zeta),r_\delta(\zeta) \big)$, oriented positively. Since $\lambda_{\delta,n}(\zeta)$ is the only eigenvalue of $P_\delta(\zeta)$ enclosed by $\gamma_\delta(\zeta)$ -- see Lemma \ref{lem:1g} -- we have
\begin{equations}\label{eq:1h}
\Pi_\delta^-(\zeta) = \dfrac{1}{2\pi i} \oint_{\gamma_\delta(\zeta)} \big(z-P_\delta(\zeta)\big)^{-1} dz.
\end{equations} 
Since $\gamma_\delta(\zeta)$ has length $O\big(r_\delta(\zeta) \big)$,
Lemma \ref{lem:1a} implies that
\begin{equations}
\Pi_\delta^-(\zeta) =  J_0(\xi)^* \cdot \dfrac{1}{2\pi i}\oint_{\gamma_\delta(\zeta)} \big(z-\Mm_\delta(\zeta)\big)^{-1} dz \cdot J_0(\xi) + \OO_{L^2_\xi}\big(r_\delta(\zeta) \big)
\\
= J_0(\xi)^* \cdot \pi_\delta^-(\zeta) \cdot J_0(\xi) + \OO_{L^2_\xi}\big(r_\delta(\zeta) \big).
\end{equations} 

3. Using \eqref{eq:1h},
\begin{equations}\label{eq:1m}
\dd{\Pi_\delta^-(\zeta)}{t} =  \dd{}{t} \left(\dfrac{1}{2\pi i}\oint_{\gamma_\delta(\zeta)} \big(z-P_\delta(\zeta)\big)^{-1} dz \right). 
\end{equations} 
Lemma \ref{lem:1a} is a direct application of the residue theorem that shows that the derivative with respect to $t$ can be interchanged with the integral -- even though $\gamma_\delta(\zeta)$ depends on $t$. Thanks to 
$\p_t P_\delta(\zeta) = \delta W_t'$, it implies
\begin{equations}
\dd{\Pi_\delta^-(\zeta)}{t} = \dfrac{1}{2\pi i} \oint_{\gamma_\delta(\zeta)} \big(z-P_\delta(\zeta)\big)^{-1} \cdot \delta W_t'  \cdot \big(z-P_\delta(\zeta)\big)^{-1} dz.
\end{equations} 
We now apply Lemma \ref{lem:1b} to get
\begin{equations}
\dd{\Pi_\delta^-(\zeta)}{t} = J_0(\xi)^* \cdot \dfrac{1}{2\pi i} \oint_{\gamma_\delta(\zeta)} \big(z-\Mm_\delta(\zeta)\big)^{-1} \cdot J_0(\xi) \delta W_t' J_0(\xi)^* \cdot  \big(z-\Mm_\delta(\zeta)\big)^{-1} dz \cdot J_0(\xi)
\\
+ \int_{\gamma_\delta(\zeta)} \OO_{L^2_\xi}\big(\delta r_\delta(\zeta)^{-1}\big) dz.
\end{equations} 
To control the remainder term, we used the estimate
\begin{equation}\label{eq:1n}
\big(z-P_\delta(\zeta)\big)^{-1} = \OO_{L^2_\xi}\big(r_\delta(\zeta)^{-1}\big),
\end{equation}
which follows from Lemma \ref{lem:1b} and the bound \eqref{eq:1j} on the resolvent of $\Mm_\delta(\zeta)$. The operator $J_0(\xi) W_t' J_0(\xi)^*$ is the $2 \times 2$ matrix
\begin{equation}
\matrice{0 & \delta \blr{\phi_+(\xi),W_t'\phi_-(\xi)} \\  \blr{\phi_-(\xi),W_t'\phi_+(\xi)} & 0 } = \matrice{0 & \ove{\var'(t)} \\ \var'(t) & 0} + \OO_{\C^2}(\xi-\pi).
\end{equation}
We recognize the matrix $\p_t \Mm_\delta(\zeta)$ modulo $O(\xi-\pi)$. Using again \eqref{eq:1j}, that $\gamma_\delta(\zeta)$ has length $O\big(r_\delta(\zeta)\big)$ and $\delta(\xi-\pi) = O\big(r_\delta(\zeta)^2\big)$, we deduce that under our assumptions, $\p_t \Pi_\delta(\zeta)$ equals
\begin{equations}
J_0(\xi)^* \cdot \dfrac{1}{2\pi i} \oint_{\gamma_\delta(\zeta)} \big(z-\Mm_\delta(\zeta)\big)^{-1} \cdot \dd{\Mm_\delta(\zeta)}{t} \cdot  \big(z-\Mm_\delta(\zeta)\big)^{-1} dz \cdot J_0(\xi) + \OO_{L^2_\xi}(\delta)
\\
= J_0(\xi)^* \cdot \dd{}{t} \left(\dfrac{1}{2\pi i} \oint_{\gamma_\delta(\zeta)} \big(z-\Mm_\delta(\zeta)\big)^{-1} dz \right) \cdot J_0(\xi) + \OO_{L^2_\xi}(\delta) 
\\
= J_0(\xi)^* \cdot \dd{\pi_\delta^-(\zeta)}{t} \cdot J_0(\xi) +  \OO_{L^2_\xi}(\delta).
\end{equations}
From \eqref{eq:1v}, $\p_t \pi_\delta^-(\zeta) = \OO_{\C^2}(1)$. This implies $\p_t \Pi_\delta(\zeta) = \OO_{L^2_\xi}(1)$.

4. Observe that $\nabla_\xi P_\delta(\zeta) = 2D_x$. 
Via arguments similar to the beginning of Step 3, 
\begin{equations}
\nabla_\xi \Pi_\delta^-(\zeta) =  \dfrac{1}{2\pi i} \oint_{\gamma_\delta(\zeta)} \big(z-P_\delta(\zeta)\big)^{-1} \cdot 2D_x  \cdot \big(z-P_\delta(\zeta)\big)^{-1} dz
\\
= J_0(\xi)^* \cdot \dfrac{1}{2\pi i} \oint_{\gamma_\delta(\zeta)} \big(z-\Mm_\delta(\zeta)\big)^{-1} \cdot J_0(\xi) 2D_x J_0(\xi)^* \cdot \big(z-\Mm_\delta(\zeta)\big)^{-1} dz \cdot J_0(\xi)
\\
+ \int_{\gamma_\delta(\xi)} \OO_{L^2_\xi}\big(r_\delta(\zeta)^{-1} \big) dz.
\end{equations}
To control the remainder term, we used $J_0(\xi) 2D_x J_0(\xi)^* = \OO_{\C^2}(1)$ and \eqref{eq:1n}. Because of Lemma \ref{lem:1c}, 
\begin{equation}
J_0(\xi) 2D_x J_0(\xi)^* = 2\matrice{ \lr{\phi_+(\xi),D_x\phi_+(\xi)} & \lr{\phi_+(\xi),D_x\phi_-(\xi)} \\ \lr{\phi_-(\xi),D_x\phi_+(\xi)} & \lr{\phi_-(\xi),D_x\phi_-(\xi)}} = \nu_\star \bst + \OO_{\C^2}(\xi-\pi).
\end{equation}
We recognize $\p_\xi \Mm_\delta(\zeta)$ modulo $O(\xi-\pi)$. Using \eqref{eq:1j} and that $\gamma_\delta(\zeta)$ has length $O\big(r_\delta(\zeta)\big)$, we argue as in the end of Step 3 and deduce that
 \begin{equations}
\nabla_\xi \Pi_\delta^-(\zeta) = J_0(\xi)^* \cdot \dd{\pi_\delta^-(\zeta)}{\xi} \cdot J_0(\xi) + O_{L^2_\xi}(1).
 \end{equations}
From \eqref{eq:1v}, $\p_\xi \pi_\delta^-(\zeta) = \OO_{\C^2}\big(\delta^{-1}\big)$. This implies that $\nabla_\xi \Pi_\delta^-(\zeta) = O\big(\delta^{-1}\big)$.

5. In Steps 2,3 and 4 we provided bounds on $\Pi_\delta^-(\zeta)$, $\p_t\Pi_\delta^-(\zeta)$ and $\nabla_\xi\Pi_\delta^-(\zeta)$ and we expanded them in terms of $\pi_\delta^-(\zeta)$, $\p_t\pi_\delta^-(\zeta)$ and $\p_\xi\pi_\delta^-(\zeta)$, respectively. These bounds lead to the expansion
\begin{equation}
\Pi_\delta^-(\zeta) \big[\nabla_\xi\Pi_\delta^-(\zeta), \p_t\Pi_\delta^-(\zeta)\big] = J_0(\xi)^* \cdot \pi_\delta^-(\zeta) \big[\nabla_\xi\pi_\delta^-(\zeta), \p_t\pi_\delta^-(\zeta)\big] \cdot J_0(\xi) + \OO_{L^2(\xi)}(1),
\end{equation}
where we used that $J_0(\xi) J_0(\xi)^*  = \Id_{\C^2}$. The operators involved have rank at most one, therefore we can take the trace and deduce that
\begin{equation}
\Tr\left(\Pi_\delta^-(\zeta) \big[\nabla_\xi\Pi_\delta^-(\zeta), \p_t\Pi_\delta^-(\zeta)\big] \right) = \Tr\left( \pi_\delta^-(\zeta) \big[\nabla_\xi\pi_\delta^-(\zeta), \p_t\pi_\delta^-(\zeta)\big] \right) + O(1).
\end{equation}
We used the cyclicity of the trace to remove $J_0(\xi)$ and $J_0(\xi)^*$ from the RHS. We conclude thanks to \eqref{eq:1o} that
\begin{equation}
B_\delta(\zeta) =  \Tr\left(\pi_\delta^-(\zeta) \big[\p_\xi\pi_\delta^-(\zeta), \p_t\pi_\delta^-(\zeta)\big] \right) + O(1).
\end{equation}
This completes the proof. \end{proof}

We now have all the ingredients to prove Theorem \ref{thm:2}. We recall that $c_1(\EE) = c_1(\EE_\delta)$ is the integral of the Berry curvature $B_\delta(\xi,t)$ over $[0,2\pi]^2$ -- see \eqref{eq:1zb} and \S\ref{sec:8.4}. Let $\epsi_0 > 0$ and $\delta_0 > 0$ be given by Lemma \ref{lem:1b}. For every $\delta \in (0,\delta_0)$ and $\epsi \in (0,\epsi_0)$,
\begin{equation}
c_1(\EE_\delta) = \dfrac{1}{2\pi i} \int_0^{2\pi}\int_{|\xi-\pi| \leq \epsi} B_\delta(\xi,t) \cdot d\xi dt + \dfrac{1}{2\pi i} \int_0^{2\pi}\int_{|\xi-\pi| \geq \epsi} B_\delta(\xi,t) \cdot d\xi dt.
\end{equation} 
Since the first integral is realized over an area $O(\epsi)$, Lemma \ref{lem:1d} yields
\begin{equation}\label{eq:1t}
\dfrac{1}{2\pi i} \int_0^{2\pi}\int_{|\xi-\pi| \leq \epsi} B_\delta(\xi,t) \cdot d\xi dt = \dfrac{1}{2\pi i} \int_0^{2\pi}\int_{|\xi-\pi| \leq \epsi} \Bb_\delta(\xi,t) \cdot d\xi dt + O(\epsi).
\end{equation}
We now \textit{fix} $\epsi$ such that the remainder term $O(\epsi)$ in \eqref{eq:1t} is smaller than $1/4$. Lemma \ref{lem:1f} shows that for $\delta$ sufficiently small,
\begin{equations}
\left|\dfrac{1}{2\pi i} \int_0^{2\pi}\int_{|\xi-\pi| \geq \epsi} B_\delta(\xi,t) d\xi dt\right| \leq \dfrac{1}{4}.
\end{equations}
We deduce that
\begin{equation}\label{eq:1p}
\left|c_1(\EE_\delta) - \dfrac{1}{2\pi i} \int_0^{2\pi} \int_{|\xi-\pi| \leq \epsi} \Bb_\delta(\xi,t) \cdot d\xi dt \right| \leq \dfrac{1}{2}.
\end{equation}

It remains to compute the integral appearing in \eqref{eq:1p}. Because of the link \eqref{eq:2u} between $\Mm_\delta(\xi,t)$ and $\MM(\xi,t)$, we have
\begin{equations}\label{eq:1q}
\dfrac{1}{2\pi i} \int_0^{2\pi}\int_{|\xi-\pi| \leq \epsi} \Bb_\delta(\xi,t) \cdot d\xi dt = \dfrac{1}{2\pi i} \int_0^{2\pi}\int_{|\xi-\pi| \leq \epsi} \dfrac{\nu_\star}{\delta} \cdot \BB\left(\dfrac{\nu_\star(\xi-\pi)}{\delta},t\right) \cdot d\xi dt
\end{equations}
We perform the substitution $\xi \mapsto \pi + \nu_\star^{-1} \delta \xi$. Since $\epsi$ is fixed, the domain of integration now approaches $\R \times [0,2\pi]$ as $\delta$ goes to zero. Using Lemma \ref{lem:2c} -- which computes the integral of $\BB(\xi,t)$ over $\R \times [0,2\pi]$ -- we deduce that \eqref{eq:1q} equals
\begin{equations}
\dfrac{1}{2\pi i} \int_0^{2\pi}\int_{|\xi| \leq \nu_F\epsi/\delta} \BB\left(\xi,t\right) \cdot \sgn(\nu_\star) d\xi dt = \dfrac{(-1)^\frac{n-1}{2}}{2\pi i} \int_0^{2\pi} \dfrac{\var'(t)}{\var(t)} dt + o(1).
\end{equations}
We take $\delta$ sufficiently small so that the term $o(1)$ is smaller than $1/4$ and \eqref{eq:1z} holds. This leads to
\begin{equation}
\left| c_1(\EE_\delta) - \dfrac{(-1)^\frac{n-1}{2}}{2\pi i} \int_0^{2\pi} \dfrac{\var'(t)}{\var(t)} dt \right| \leq \dfrac{3}{4}.
\end{equation}
Since $c_1(\EE_\delta)$ is an integer, the proof of Theorem \ref{thm:2} is complete:
\begin{equation}
c_1(\EE) = c_1(\EE_\delta) = \dfrac{(-1)^\frac{n-1}{2}}{2\pi i} \int_0^{2\pi} \dfrac{\var'(t)}{\var(t)} dt.
\end{equation}

\appendix

\section{}

\subsection{Dirac points}\label{sec:7.5} We deal here with operators $D_x^2+\VVV$, where $\VVV$ is a one-periodic potential satisfying an additional symmetry: $\VVV(x+1/2) = \VVV(x)$. In particular, the fundamental cell is $[0,1]$ and Brillouin zone is $[0,2\pi]$. Recall that 
\begin{equations}
L^2_{\xi,\ev} = \left\{ u \in L^2_\loc : u(x+1/2,\xi) = e^{i\xi} u(x)\right\}, 
\\
 L^2_{\xi,\od} = \left\{ u \in L^2_\loc : u(x+1/2,\xi) = -e^{i\xi} u(x)\right\}.
\end{equations}
The operator $D_x^2+\VVV$ leaves these spaces invariant; we denote by $\mu_{\ev,1}(\xi) \leq \dots \leq \mu_{\ev,j}(\xi)$ and $\mu_{\od,1}(\xi) \leq \dots \leq \mu_{\od,j}(\xi)$  its eigenvalues on $L^2_{\xi,\ev}$ and $L^2_{\xi,\od}$, respectively. 

\begin{lem}\label{lem:2m} The functions $\mu_{\ev,j}$ and $\mu_{\od,j}$ are monotonous on $[0,2\pi]$. Specifically,
\begin{equation}
j \text{ odd } \Rightarrow \ \systeme{ \mu_{\ev,j} \text{ increases on } [0,2\pi] \\ \mu_{\od,j} \text{ decreases on } [0,2\pi] }, \ \ j  \text{ even } \Rightarrow \ \systeme{ \mu_{\ev,j} \text{ decreases on } [0,2\pi] \\ \mu_{\od,j} \text{ increases on } [0,2\pi] }.
\end{equation}
In addition, for every $\xi \in (0,2\pi)$ and $j \geq 1$, the eigenvalues $\mu_{\ev,j}(\xi)$ and $\mu_{\od,j}(\xi)$ of $D_x^2 + \VVV$ on $L^2_{\xi,\ev}$, $L^2_{\xi,\od}$ are simple.
\end{lem}

\begin{proof} 1. We look at $D_x^2+\VVV$ as a $1/2$-periodic operator; the associated fundamental cell is $[0,1/2]$ and the corresponding Brillouin zone is $[0,4\pi]$. For $\xi \in [0,4\pi]$, let $T(\xi)$ formally equal to $D_x^2 + \VVV$, acting on the space of $\xi$-quasiperiodic functions
\begin{equation}
\tL^2_\xi = \left\{ u \in L^2_\loc : u(x+1/2,\xi) = e^{i\frac{1}{2}\xi} u(x)\right\}
\end{equation}
(the factor $\frac{1}{2}$ accounts for the period $1/2$ of the lattice $\Z/2$).
This space equals $L^2_{\xi,\ev}$ when $\xi \in [0,2\pi]$. Moreover,
\begin{equation}
e^{i\frac{1}{2}\xi} = -e^{i\frac{1}{2}(\xi-2\pi)}.
\end{equation}
Therefore when $\xi \in [2\pi,4\pi]$, $\tL^2_\xi = L^2_{\xi-2\pi,\od}$. We deduce that
\begin{equations}\label{eq:3n}
\xi \in [0,2\pi] \ \Rightarrow \ T(\xi) \text{ is equal to } D_x^2 + \VVV : L^2_{\xi,\ev} \rightarrow L^2_{\xi,\ev};
\\
\xi \in [2\pi,4\pi] \ \Rightarrow \ T(\xi) \text{ is equal to } D_x^2 + \VVV  : L^2_{\xi-2\pi,\od} \rightarrow L^2_{\xi-2\pi,\od}.
\end{equations}

2. Let $\tE_1(\xi) \leq \dots \leq \tE_j(\xi) \leq \dots$ be the $\tL^2_\xi$-eigenvalues of $T(\xi)$. Because of \eqref{eq:3n},
\begin{equation}\label{eq:6k}
\tE_j(\xi) = \systeme{\mu_{\ev,j}(\xi) & \text{ if } \xi \in [0,2\pi] \\ \mu_{\od,j}(\xi-2\pi) & \text{ if } \xi \in [2\pi,4\pi].}
\end{equation}
For $\xi \in (0,2\pi) \cup (2\pi,4\pi)$, the $\tE_\ell(\xi)$ are simple eigenvalues -- see \cite[Theorem XIII.89]{RS}. It follows that for $\xi \in (0,2\pi)$, $L^2_{\xi,\ev}$- and $L^2_{\xi,\od}$-eigenvalues of $D_x^2+\VVV$ are simple. That same theorem implies that if $\ell$ is odd (resp. even) then $\tE_\ell$ increases 
(resp. decreases) on $[0,2\pi]$ and $\tE_\ell$ decreases 
(resp. increases) on $[2\pi,4\pi]$. We deduce from \eqref{eq:6k} that $\xi \in [0,2\pi] \mapsto \mu_{\ev,\ell}(\xi)$ and $\xi \in [0,2\pi] \mapsto \mu_{\od,\ell}(\xi)$ are monotonous. More specifically,
\begin{equation}
\ell \text{ odd } \Rightarrow \ \systeme{ \mu_{\ev,\ell} \text{ increases on } [0,2\pi] \\ \mu_{\od,\ell} \text{ decreases on } [0,2\pi] }, \ \ \ \ell \text{ even } \Rightarrow \ \systeme{ \mu_{\ev,\ell} \text{ decreases on } [0,2\pi] \\ \mu_{\od,\ell} \text{ increases on } [0,2\pi] }.
\end{equation}
This completes the proof.\end{proof} 

\begin{center}
\begin{figure}
{\caption{The graph on the left represents the dispersion curves of $D_x^2+\VVV$ as a one-periodic operator. The $L^2_{\xi,\ev}$-curves are plotted in blue and the $L^2_{\xi,\od}$-curves are plotted in red. To obtain the dispersion curves of $D_x^2+\VVV$ as a $1/2$-periodic operator, it suffices to concatenate the $L^2_{\xi,\ev}$-curves with the $L^2_{\xi,\od}$-curves.}\label{fig:6}}
{\includegraphics{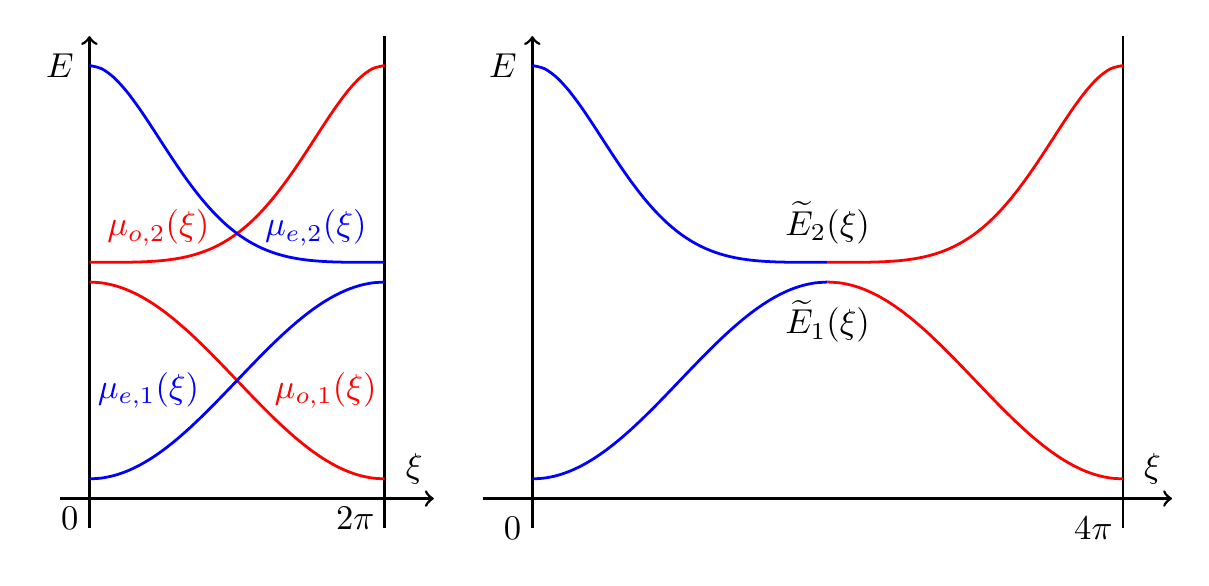}}
\end{figure}
\end{center}

For a pictorial representation of the proof of Lemma \ref{lem:2m}, see Figure \ref{fig:6}. Because of the monotonicity property \eqref{eq:3h}, we see that the $L^2_{\xi,\ev}$ and $L^2_{\xi,\od}$-eigenvalues  of $D_x^2+\VVV$ relate to the $L^2_\xi$-eigenvalues $E_1(\xi) \leq \dots \leq E_\ell(\xi) \leq \dots $ of $D_x^2+\VVV$ via 
\begin{equation}
E_{2j-1}(\xi) = \systeme{\mu_{\ev,j}(\xi) \text{ if } \xi \in [0,\pi] \\ \mu_{\od,j}(\xi) \text{ if } \xi \in [\pi,2\pi]}, \ \ \ \ E_{2j}(\xi) = \systeme{\mu_{\od,j}(\xi) \text{ if } \xi \in [0,\pi] \\ \mu_{\ev,j}(\xi) \text{ if } \xi \in [\pi,2\pi]}.
\end{equation}
See the left graph of Figure \ref{fig:6}. Moreover, $\mu_{\ev,N}(\xi) = \mu_{\od,N}(2\pi-\xi)$. This comes from $\VVV$ real-valued: for any $E$,
\begin{equation}\label{eq:3q}
u \in L^2_{\xi,\ev}, \ \ \big(D_x^2+\VVV-E\big) u =0 \ \Rightarrow \ \ou \in L^2_{2\pi-\xi,\od}, \ \ (D_x^2+\VVV-E) \ou =0.
\end{equation}

\begin{proof}[Proof of Lemma \ref{lem:1u}] 1. Write $n=2N-1$ and $E_\star = E_n(\pi)$. From \eqref{eq:3q}, we see that $\mu_{\ev,N}(\pi) = \mu_{\od,N}(\pi)$, that is, the $n$-th and $n+1$-th dispersion curves of $D_x^2+\VVV$ intersect at $\pi$. To check that this intersection always corresponds to a Dirac point, we must observe that $\p_\xi \mu_{\ev,N}(\pi) \neq 0$.

2. Assume that $\p_\xi \mu_{\ev,N}(\pi) = 0$; we look for a contradiction. The equation $\mu_{\ev,N}(\xi) = E_\star$ has a zero of algebraic multiplicity $2$. The function $\xi \in (0,2\pi) \mapsto \mu_{\ev,N}(\xi)$ is real analytic.  Indeed, it represents $L^2_{0,\ev}$-eigenvalues of $(D_x+\xi)^2 + \VVV$, which are simple. It follows that $\mu_{\ev,N}$ has an analytic extension to an open set $U \subset \C$ containing $(0,2\pi)$.

Let $E \in \R \setminus \{E_\star\}$, sufficiently close to $E_\star$. From Rouch\'e's theorem, the equation $\mu_{\ev,N}(\xi) = E$ has at least two solutions $\xi_1$ and $\xi_2$ in $U$. Since $\mu_{\ev,N}(\xi)$ is strictly monotonous and continuous on $(0,2\pi)$, the intermediate value theorem implies that (say) $\xi_1 \in (0,2\pi) \setminus \{\pi\}$; and $\xi_2 \notin \R$. Let $\xi_3 = 2\pi-\xi_1$. Note that $\xi_1, \ \xi_2$ and $\xi_3$ are pairwise distinct; and because of \eqref{eq:3q}, $\mu_{\od,N}(\xi_3) = E$. 

3. For $\xi \in (0,2\pi)$, the $\mu_{\ev,\ell}(\xi)$'s are the $L^2_{0,\ev}$-eigenvalues of $(D_x+\xi)^2 + \VVV$. Using the unique continuation principle, this remains true for $\xi \in U$. An analogous statement holds for $\mu_{\od,\ell}(\xi)$. From Step 2, we deduce that there exist $0 \not\equiv \psi_j \in L^2_{\xi_j}$, $j=1,2,3$ such that
\begin{equation}\label{eq:3m}
\big( D_x^2+\VVV-E \big) \psi_j = 0.
\end{equation}
If $(i,j,k)$ are pairwise distinct then so are $(\xi_i, \xi_j, \xi_k)$; and $L^2_{\xi_i} \cap (L^2_{\xi_j} \oplus L^2_{\xi_k})  = \{0\}$. We conclude that $\psi_1, \psi_2, \psi_3$ are three linearly independent
solutions of the second order ODE \eqref{eq:3m}. This is a contradiction. We deduce that $\p_\xi\mu_{\ev,N}(\pi) \neq 0$, which ends the proof.\end{proof}

A consequence of Lemma \ref{lem:2m} and of the proof of Lemma \ref{lem:1u}  is that under the notations of Lemma \ref{lem:1c},  $\nu_\star = \lr{D_x\phi_+^\star,\phi_+^\star} > 0$ when $n = 1 \mod 4$ and $\nu_\star < 0$ when $n = 3 \mod 4$. Equivalently,
\begin{equation}\label{eq:6q}
\sgn(\nu_\star) = (-1)^{\frac{n-1}{2}}.
\end{equation}

\subsection{Examples}\label{app:7} Here we prove the results of \S\ref{sec:8}. We will need the following lemma.

\begin{lem}\cite[Lemma 3.1]{DFW}\label{lem:2y} Let $T$ be a selfadjoint operator on a Hilbert space $\HH$ with discrete spectrum. Assume that there exist $\psi$ in the domain of $T$, $E \in \R$ and $\eta > 0$ such that 
\begin{equations}
|\psi|_\HH = 1\quad {\rm and}\  \ |(T-E)\psi|_\HH \leq \eta.
\end{equations}
Then, $T$ has a eigenvalue $\lambda$ with $|\lambda-E| \leq \eta$. 

Furthermore, if $T$ has no other eigenvalue in the interval $[E-C\eta,E+C\eta]$ for some $C \geq 1$ then $T$ has an eigenvector $\phi \in \ker_\HH(T-\lambda)$ with $|\psi-\phi|_\HH \leq C^{-1}$.
\end{lem}

The function $\psi$ is usually called a ($\eta$-accurate) quasimode.
Lemma \ref{lem:2y} is a consequence of variational principles for selfadjoint eigenvalue problems. It can dramatically refines results at barely any cost once one has (a) an \textit{accurate quasimode} and (b) an \textit{eigenvalue separation estimate}. See the application \cite[Corollary 1]{DFW} to get full expansions of defect states for the $\PP_{\delta,\star}$-spectral problem. 

\begin{proof}[Proof of Lemma \ref{lem:1f}] This lemma deals with small potentials. We verify the assumptions $\oA$ and $\oB$ using perturbative analysis.  The general idea is (a) use Lemma \ref{lem:1m} to estimate Bloch eigenpairs with a crude error bound; (b) use Lemma \ref{lem:2y} to refine these estimates. In the proof we suppose -- without loss of generalities -- that $|\VV|_\infty \leq 1$ and $|\WW|_\infty \leq 1$.

1. We first assume that $n = 1 \mod 4$. The operator $D_x^2$ has eigenvalues $\ell^2\pi^2$ on $L^2_\pi$, $\ell \in 2\N+1$, each of them with double multiplicity. The operator $D_x^2+V$ has eigenvalues
\begin{equation}
\lambda_{0,1}(\pi,t) = \lambda_{0,2}(\pi,t) < \lambda_{0,3}(\pi,t) = \lambda_{0,4}(\pi,t) < \dots
\end{equation}
see Lemma \ref{lem:1u} (technically speaking, these eigenvalues do not depend on $t$). Lemma \ref{lem:1m} applied to $\VVV = V$, $\WWW = 0$ implies that for $\epsi$ sufficiently small,
\begin{equation}\label{eq:2v}
\ell \in [1,n+1] \ \text{ odd } \ \ \Rightarrow \ \ 
\lambda_{0,\ell}(\pi,t) = \lambda_{0,\ell+1}(\pi,t) = \ell^2\pi^2+ O(\epsi).
\end{equation}
We set $E_\star = \lambda_{0,n}(\pi,t)$.

Let $f_\pm(x) = e^{\pm i\pi nx}$. The function $f_+$ belongs to $L^2_{\pi,\ev}$ (because $n = 1 \mod 4$) and
\begin{equation}
\big(D_x^2+V - E_\star\big) f_+ = \big(D_x^2-n^2\pi^2\big) f_+ + O_{L^2_{\pi,\ev}}(\epsi) = O_{L^2_{\pi,\ev}}(\epsi).
\end{equation}
Hence, $f_+$ is a $O(\epsi)$-precise quasimode for $D_x^2+V$ on $L^2_{\pi,\ev}$. In particular $D_x^2+V$ has a $L^2_{\pi,\ev}$-eigenvalue equal to $n^2\pi^2 + O(\epsi)$ -- see Lemma \ref{lem:2y}. Because of \eqref{eq:2v}, this eigenvalue must be $E_\star$. 

The eigenvalues of $D_x^2+V$ on $L^2_{\pi,\ev}$ are simple -- see Lemma \ref{lem:2m}. The equation \eqref{eq:2v} implies that $E_\star$ is the only $L^2_{\pi,\ev}$ eigenvalue of $D_x^2+V$ in $[n^2\pi^2-1,n^2\pi^2+1]$ for $\epsi$ sufficiently small. We deduce from Lemma \ref{lem:2y} applied to $\eta = O(\epsi)$, $C^{-1} = O(\epsi)$ that a corresponding $L^2_{\pi,\ev}$-eigenvector is $f_+ + O_{L^2_{\pi,\ev}}(\epsi)$. Since $V$ is real-valued, taking the complex conjugate produces a $L^2_{\pi,\od}$-eigenvector. Hence, a Dirac eigenbasis for the Dirac point $(\pi,E_\star)$ is 
\begin{equation}
\big( \phi_+^\star,\phi_-^\star \big) = (f_+,f_-) + O_{L^2_\pi}(\epsi).
\end{equation}

In particular, 
\begin{equations}\label{eq:9a}
\var(t) = \blr{\phi_-^\star,W_t\phi_+^\star} = \epsi \int_0^1 \WW\big(x+t/(2\pi) \big) e^{2i\pi nx} dx + O\left(\epsi^2\right)
\\
= \epsi \int_0^1 \WW(x) e^{2i\pi nx - int} dx + O\left(\epsi^2\right) = \epsi\ove{\hWW_n} \cdot e^{-int} + O\left(\epsi^2\right).
\end{equations}
Since we assumed that $\hWW_n \neq 0$, $\var$ does not vanish for small enough $\epsi$: $\oB$ holds. Furthermore, the degree of $\var$ is $-n = (-1)^{\frac{n+1}{2}} \cdot n$, because $n = 1 \mod 4$.

When $n =3 \mod 4$ then $f_+ \in L^2_{\pi,\od}$ and $f_- \in L^2_{\pi,\ev}$ instead. Therefore we only need to interchange $f_+$ and $f_-$. A Dirac eigenbasis is
\begin{equation}
\big( \phi_+^\star,\phi_-^\star \big) = (f_-,f_+) + O_{L^2_\pi}(\epsi).
\end{equation}
Since $\WW$ is real-valued, a calculation similar to \eqref{eq:9a} implies that $\var(t)$ equals 
\begin{equation}
\epsi \int_0^1 \WW(x) e^{-2i\pi nx + int} dx + O\left(\epsi^2\right) 
= \epsi e^{int} \int_0^1 \WW(x) e^{-2i\pi nx} dx + O\left(\epsi^2\right) = \epsi \hWW_n \cdot e^{int} + O\left(\epsi^2\right).
\end{equation}
For small enough $\epsi$, $\var(t)$ does not vanish: $\oB$ holds. Here, the degree of $\var$ is $n = (-1)^{\frac{n+1}{2}} \cdot n$, since $n = 3 \mod 4$.

2. We prove that for $s \in (0,1]$ and $\epsi$ sufficiently small, the $n$-th and $n+1$-th $L^2_\pi$-eigenvalues of $D_x^2 + V + s W_t$ satisfy
\begin{equation}
\lambda_{s,n}(\pi,t) < E_\star < \lambda_{s,n+1}(\pi,t).
\end{equation}
As in Step 1, Lemma \ref{lem:1m} estimates crudely the eigenvalues of $D_x^2 + V + sW_t$: for $\epsi$ sufficiently small, for each odd $\ell \in [1,n+1]$,
\begin{equation}\label{eq:2s}
\lambda_{s,\ell}(\pi,t) = \ell^2\pi^2+ O(\epsi) , \ \lambda_{s,\ell+1}(\pi) = \ell^2\pi^2+ O(\epsi).
\end{equation}
When $\ell = n$, we refine this estimate by constructing an accurate quasimode $g$. Let 
\begin{equations}\label{eq:9b}
(a,b) \neq (0,0) \in \C^2; \ \ \ \  u \in C^\infty(\R,\C) \cap L^2_\pi \ \text{ with }  \ u = O_{L^2_\pi}(\epsi); \\ \lambda \in \R \ \text{ with } \ \lambda = O(\epsi).
\end{equations}
Set $g = a \phi^\star_+ + b\phi^\star_- + su$ and observe that
\begin{equation}
\big( D_x^2 + V + s W_t - E_\star - s\lambda \big) g = \big( D_x^2+V-E_\star \big) su + s (W_t- \lambda) \big( a \phi^\star_+ + b\phi^\star_- \big) + O_{L^2_\pi}(s^2\epsi^2).
\end{equation}
For $g$ to be an accurate quasimode, the leading term of the RHS must vanish:
\begin{equation}\label{eq:3k}
\big(D_x^2+V-E_\star \big) u + (W_t- \lambda)(a \phi^\star_+ + b\phi^\star_-) =0.
\end{equation}
We prove that \eqref{eq:3k} admits a solution $(a,b,u,\lambda)$ satisfying \eqref{eq:9b}. We look at \eqref{eq:3k} as a inhomogeneous equation for $u$. For a solution to exist, it suffices to have the inhomogeneous term $(W_t- \lambda)(a \phi^\star_+ + b\phi^\star_-)$ orthogonal to $\ker_{L^2_\pi}(D_x^2+V-E_\star) = \C \phi_+^\star \oplus \C \phi_-^\star$. We obtain a system of equations:
\begin{equation}\label{eq:3j}
\systeme{\blr{\phi_+^\star,W_t\phi_+^\star}a + b\blr{\phi_+^\star,W_t\phi_-^\star} - \lambda a = 0 \ \\ \blr{\phi_-^\star,W_t\phi_+^\star}a + b\blr{\phi_-^\star,W_t\phi_-^\star} - \lambda a = 0.}
\end{equation}
Note that $\blr{\phi_+^\star,W_t \phi_+^\star} = \blr{\phi_-^\star,W_t \phi_-^\star} = 0$ -- see Lemma \ref{lem:2z}. The system \eqref{eq:3j} takes the form of a $2 \times 2$-eigenvalue problem:
\begin{equation}
\left(\matrice{0 & \ove{\var(t)} \ \\ \var(t) & 0} - \lambda \right) \matrice{a \\ b} = 0.
\end{equation}
It has non-trivial solutions precisely when $\lambda = \pm |\var(t)|$. Because $\var(t) = O(\epsi)$ (see Step 1), $\lambda = O(\epsi)$. Since $W_t = O(\epsi)$, \eqref{eq:3k} admits solution $u_\pm = O_{L^2_\pi}(\epsi)$ -- corresponding respectively to $\lambda = \pm |\var(t)|$. Hence, we constructed $g_\pm \in L^2_\pi$ such that
\begin{equation}
(D_x^2 + V + s W_t - E_\star \pm s|\var(t)|) g_\pm = O_{L^2_\pi}(s^2\epsi^2).
\end{equation}
Lemma \ref{lem:2y} (with $\eta = O(s^2 \epsi^2)$) implies that $D_x^2 + V + s W_t$ has two eigenvalues given by $E_\star \pm s|\var(t)| + O(s^2 \epsi^2)$. These eigenvalues must be $\lambda_{s,n}(\pi,t)$ and $\lambda_{s,n+1}(\pi,t)$ because of \eqref{eq:2s} and of $E_\star = n^2 \pi^2 + O(\epsi)$. We deduce that for $s \in (0,1]$ and $\epsi$ small enough,
\begin{equation}
\lambda_{s,n}(\pi,t) < E_\star < \lambda_{s,n+1}(\pi,t).
\end{equation}
In particular, $\oA$ holds with $E(s,t) \equiv E_\star$.
\end{proof}

\begin{proof}[Proof of Lemma \ref{lem:2n}] 1. We first assume that $n = 1 \mod 4$. The strategy is similar to that of Lemma \ref{lem:1f}. By Lemma \ref{lem:1m}, for $\epsi$ small enough, for each odd $\ell \in [1,n+1]$,
\begin{equation}\label{eq:2z}
\lambda_{0,\ell}(\pi,t) = \lambda_{0,\ell+1}(\pi,t) = \ell^2\pi^2+ O\left(\epsi^2\right).
\end{equation}
Set $E_\star = \lambda_{0,n}(\pi,t) = \lambda_{0,n}(\pi,0) = n^2\pi^2+O(\epsi^2)$. The operator $D_x^2+V$ has a Dirac point at $(\pi,E_\star)$. We  use Lemma \ref{lem:2y} to construct an approximation of a Dirac eigenbasis $(\phi_+^\star,\phi_-^\star)$. Set 
\begin{equation}
f_+(x) \de e^{i\pi nx} \left( 1 - \dfrac{\epsi^2}{8\pi^2(m-n)} \left( \dfrac{e^{2i\pi (m-n) x}}{m} + \dfrac{ e^{-2i\pi (m-n) x}}{m-2n} \right)\right), \ \ f_-(x) \de \ove{f_+(x)}.
\end{equation}
Since $n = 1 \mod 4$, $f_+ \in L^2_{\pi,\ev}$, $f_- \in L^2_{\pi,\od}$. Moreover, a computation shows that
\begin{equation}
\big( D_x^2+ V-n^2\pi^2 \big) f_\pm = O_{L^2_\pi}(\epsi^3), \ \ |f_\pm|_{L^2_\pi}  = 1 + O\left(\epsi^3\right). 
\end{equation}
Therefore $D_x^2 + V$ admits $L^2_{\pi,\ev}$ and $L^2_{\pi,\od}$-eigenvectors $O\left(\epsi^3\right)$-close to $f_\pm$, with energy $O(\epsi^3)$ close to $n^2\pi^2$, see Lemma \ref{lem:2y} and the first step in the proof of Lemma \ref{lem:1f}. Because of \eqref{eq:2z}, we deduce that a Dirac eigenbasis is
\begin{equation}
(\phi_+^\star,\phi_-^\star) = (f_+,f_-) + O_{L^2_\pi}\left(\epsi^3\right).
\end{equation}

We can now estimate $\var(t)$: observe that
\begin{equations}
\ove{\phi^\star_-(x)} W_t(x) \phi^\star_+(x) = W\big(x+t/(2\pi)\big) f_+(x)^2 + O\left(\epsi^7\right)
\\
= \epsi^4 e^{2i\pi n x} \left( 1 - \dfrac{\epsi^2}{8\pi^2(m-n)} \left( \dfrac{e^{2i\pi (m-n) x}}{m} + \dfrac{ e^{-2i\pi (m-n) x}}{m-2n} \right)\right)^2 \cos(2\pi m x+t)  + O\left(\epsi^7\right)
\\
 = \epsi^4 \left( e^{2i\pi nx} - \dfrac{\epsi^2}{4\pi^2(m-n)} \left( \dfrac{e^{2i\pi m x}}{m} + \dfrac{ e^{-2i\pi (m-2n) x}}{m-2n} \right)\right) \cos(2\pi m x+t) + O\left(\epsi^7\right).
\end{equations}
Averaging this over $[0,1]$, we get:
\begin{equation}\label{eq:9d}
\var(t) = -\dfrac{\epsi^6}{8\pi^2(m-n)m} e^{-imt}+ O\left(\epsi^7\right).
\end{equation}
For $\epsi$  sufficiently small, $\var$ does not vanish and its degree is $m = (-1)^{\frac{n+1}{2}} \cdot m$, because $n = 1 \mod 4$. If now $n=3 \mod 4$, then we need to change $f_+$ and $f_-$ (as in the proof of Lemma \ref{lem:1f}). We end up with
\begin{equation}\label{eq:9e}
\var(t) = -\dfrac{\epsi^6}{8\pi^2(m-n)m} e^{imt}+ O\left(\epsi^7\right).
\end{equation}
Therefore, the winding number \eqref{eq:1z} equals $m =(-1)^{\frac{n+1}{2}} \cdot m$. In particular, $\oB$ holds for sufficiently small $\epsi$.

2. The exact same argument as in Step 2 in the proof of Lemma \ref{lem:1f} shows that
\begin{equation}
\lambda_{s,n}(\pi,t) = E_\star - s|\var(t)|+O\left(s^2\epsi^8\right), \ \ \lambda_{s,n+1}(\pi,t) = E_\star + s|\var(t)|+O\left(s^2\epsi^8\right).
\end{equation}
From \eqref{eq:9d} and \eqref{eq:9e}, we deduce that
\begin{equation}
|\lambda_{s,n}(\pi,t) - \lambda_{s,n+1}(\pi,t)| \geq \dfrac{\epsi^6}{2\pi^2} + O\left(s^2\epsi^8\right).
\end{equation}
In particular, $\oA$ holds for sufficiently small $\epsi$.
\end{proof}

\subsection{Eigenvalue estimates}\label{sec:7.2} Recall that the operator $P_{\delta}(\xi,t)$ is $D_x^2 + V + \delta W_t$ acting on $L^2_\xi$ and that $\lambda_{\delta,\ell}(\xi,t)$ is the $\ell$-th eigenvalue of $P_{\delta}(\xi,t)$.

\begin{lem}\label{lem:1h} There exist $\delta_0 > 0$ and $C > 0$ such that
\begin{equations}
\delta \in (0,\delta_0), \ \ |\xi-\pi| \leq \delta_0, \ \ z \in \BBb\big(\mu_\delta^-(\zeta),r_\delta(\zeta)\big), \ \ t \in \R \\
 \Rightarrow \ P_{0}(\xi,t) : \VVV(\xi)^\perp \rightarrow \VVV(\xi)^\perp \text{ is invertible and } \left((P_{0}(\zeta)-z)|_{\VVV(\xi)^\perp}\right)^{-1} = O_{\VVV(\xi)^\perp}(1). 
\end{equations}
\end{lem}

\begin{proof} Fix $\zeta = (\xi,t)$. The spectrum of $P_{0}(\zeta) : \VVV(\xi)^\perp \rightarrow \VVV(\xi)^\perp$ is precisely 
\begin{equation}
\Sigma_{L^2_\xi}\big(P_{0}(\zeta)\big) \setminus \big\{\lambda_{0,n}(\zeta), \lambda_{0,n+1}(\zeta)\big\}
\end{equation}
Because of the monotonicity properties of $\lambda_{0,n}(\zeta)$ and $\lambda_{0,n+1}(\zeta)$  -- see \eqref{eq:3h} -- there exist $\epsi_0, \ \epsi_1 > 0$ such that for $\xi \in [\pi/2,3\pi/2]$,
\begin{equation}\label{eq:8t}
\lambda_{0,n-1} (\zeta) < E_\star-2\epsi_0 < \lambda_{0,n}(\zeta) \leq E_\star \leq \lambda_{0,n+1}(\zeta) < E_\star+2\epsi_1 < \lambda_{0,n+2}(\zeta).
\end{equation}
Let $\epsi = \min(\epsi_0,\epsi_1)$. The equation \eqref{eq:8t} implies that
 the distance between $\BBb(E_\star,\epsi)$ and the $\VVV(\xi)^\perp$-spectrum of $P_{0}(\zeta)$ is at least $\epsi$. It follows that
\begin{equation}
z \in \BBb(E_\star,\epsi) \ \Rightarrow \ \left\| \left((P_{0}(\zeta)-z)|_{\VVV(\xi)^\perp}\right)^{-1} \right\|_{\VVV(\xi)^\perp} \leq \epsi^{-1}.
\end{equation}
To conclude, we simply observe that $\BBb(\mu_\delta^-(\zeta),r_\delta(\zeta)) \subset \BBb(E_\star,\epsi)$ when $\delta$ and $|\xi-\pi|$ are small enough because
\begin{equation}
|E_\star - \mu_\delta^-(\zeta)| \leq \nu_F |\xi-\pi| + \delta |\var(t)|.
\end{equation}
This completes the proof. \end{proof}

\begin{lem}\label{lem:1g} There exists $\delta_0 > 0$ such that 
\begin{equations}
\delta \in (0,\delta_0), \ |\xi-\pi| \leq \delta_0, \ t \in \R 
\\
 \Rightarrow \ \lambda_{\delta,n}(\zeta) \in \BBb\big(\mu_\delta^-(\zeta), r_\delta(\zeta)\big), \ \lambda_{\delta,n+1}(\zeta) \in \BBb\big(\mu_\delta^+(\zeta), r_\delta(\zeta)\big).
\end{equations}
\end{lem}

\begin{proof} 1. We show that $P_{\delta}(\zeta)$ has an eigenvalue in each of the balls $\BBb(\mu_\delta^-(\zeta),r_\delta(\zeta))$ and $\BBb(\mu_\delta^+(\zeta),r_\delta(\zeta))$. Because of Lemma \ref{lem:1m}, these must be the $n$-th and $n+1$-th eigenvalues of $P_{\delta}(\zeta)$. This would prove the lemma.

2. According to Lemma \ref{lem:1b}, the operator $P_{\delta}(\zeta)-z$ is invertible for $z \in \BBb(\mu_\delta^-(\zeta),r_\delta(\zeta))$. It follows that
\begin{equation}
\# \Sigma_{L^2_\xi}\big(P_{\delta}(\zeta)\big) \cap \BBb\big(\mu_\delta^-(\zeta),r_\delta(\zeta)\big) = \trace \left[\dfrac{1}{2\pi i} \oint_{\gamma_\delta^-(\zeta)} \big( z-P_{\delta}(\zeta) \big)^{-1} dz \right],
\end{equation}
where $\gamma_\delta^-(\zeta) = \p \BBb(\mu_\delta^-(\zeta),r_\delta(\zeta))$. Again from Lemma \ref{lem:1b},
\begin{equation}
\dfrac{1}{2\pi i} \oint_{\gamma_\delta^-(\zeta)} \big( z - P_{\delta}(\zeta) \big)^{-1} dz = \Pi_0(\xi)^* \cdot \dfrac{1}{2\pi i}  \oint_{\gamma_\delta^-(\zeta)} \big( z - \Mm_\delta(\zeta) \big)^{-1} dz \cdot \Pi_0(\xi) + \OO_{L^2_\xi}\big(r_\delta(\zeta)\big).
\end{equation}
Because of Lemma \ref{lem:1m}, the LHS is of rank at most $2$. The leading term in the RHS is of rank $1$. Therefore the remainder term is at most of rank $3$ and we can take the trace. It yields
\begin{equation}
\Tr\left[\dfrac{1}{2\pi i} \oint_{\gamma_\delta^-(\zeta)} \big( z - P_{\delta}(\zeta) \big)^{-1} dz \right] = 1 + O\big(r_\delta(\zeta)\big).
\end{equation}
We deduce that $P_{\delta}(\zeta)$ has exactly one eigenvalue in $\BBb(\mu_\delta^-(\zeta),r_\delta(\zeta))$. The same holds for $\BBb(\mu_\delta^+(\zeta),r_\delta(\zeta))$, and the conclusion follows.
\end{proof}

\subsection{A contour integration lemma}

Let $U$ be an open subset of $\R^N$, $\UU$ an open subset of $\C \times U$ and $F : \UU \rightarrow \C$. We say that $F$  is a smooth family of meromorphic functions  for $\zeta \in U$ if there exist $S, \ T \in C^\infty(\C \times U,\C)$ such that $\p_\oz S = \p_\oz T = 0$ and
\begin{equation}
(z,\zeta) \in \UU \ \Rightarrow \ F(z,\zeta) = \dfrac{S(z,\zeta)}{T(z,\zeta)}.
\end{equation}
In \S\ref{sec:5}, we use thoroughly the following lemma. The proof is an immediate consequence of the residue theorem and is omitted.

\begin{lem}\label{lem:1a} Let $F$ be as above. Let $(\zeta,\te) \in U \times \Ss^1 \mapsto \gamma(\zeta,\te) \in \C$ be a piecewise $C^1$ contour such that $(\gamma(\zeta,\te),\zeta)$ always belong to $\UU$. 
Then
\begin{equation}
I : \zeta \in U \mapsto \dfrac{1}{2\pi i} \oint_{\gamma(\zeta)} F(z,\zeta) dz
\end{equation}
is smooth and for $j=1, \dots, N$,
\begin{equation}
\dd{I(\zeta)}{\zeta_j} = \dfrac{1}{2\pi i} \oint_{\gamma(\zeta)} \dd{F(z,\zeta)}{\zeta_j} dz.
\end{equation}
\end{lem}

\subsection{Berry curvature of a family of matrices.}\label{sec:7.4}  Here we prove Lemma \ref{lem:2c} about the Berry curvature and Chern class of the low energy eigenbundle of the matrix 
\begin{equation}
\MM(\zeta) = \dfrac{1}{\delta} \left(\Mm_\delta\left(\pi+\dfrac{\delta \xi}{\nu_\star},t\right)-E_\star\right) = \matrice{\xi & \ove{\var(t)} \ \\ \var(t) & -\xi}.
\end{equation}

\begin{proof}[Proof of Lemma \ref{lem:2c}] 1. We use $\zeta =(\xi,t)$. We recall that $\MM(\zeta)$ has eigenvalues $\pm \sqrt{\xi^2 + |\var(t)|^2}$. The projector $\pi^-(\zeta)$ to the negative eigenspace of $\MM(\zeta)$ is 
\begin{equation}
\pi^-(\zeta) = \dfrac{1}{2\pi i} \int_{\gamma(\zeta)} \big( z -\MM(\zeta) \big)^{-1} dz.
\end{equation}
where $\gamma(\zeta)$ is the disk centered at $-\sqrt{\xi^2 + |\var(t)|^2}$, of radius $\var_0 = \min_{t\in [0,2\pi]} |\var(t)|$. Lemma \ref{lem:1a} shows that
\begin{equations}
\dd{\pi^-(\zeta)}{t} = \dfrac{1}{2\pi i} \int_{\gamma(\zeta)} \big( z -\MM(\zeta) \big)^{-1}   \matrice{0 & \ove{\var'(t)} \ \\ \var'(t) & 0}\big( z -\MM(\zeta) \big)^{-1}dz,
\\
\dd{\pi^-(\zeta)}{\xi} = \dfrac{1}{2\pi i} \int_{\gamma(\zeta)} \big( z -\MM(\zeta) \big)^{-1}   \matrice{1 & 0 \ \\ 0 & -1}\big( z -\MM(\zeta) \big)^{-1}dz.
\end{equations}
Since the eigenvalues of $\MM(\zeta)$ are at least $\var_0$-distant from $\gamma(\zeta)$, we can control $\p_t \pi^-(\zeta)$ and $\p_\xi \pi^-(\zeta)$ using the spectral theorem. It shows that uniformly in $\zeta \in \R \times [0,2\pi]$, $\p_t \pi^-(\zeta) = \OO_{\C^2}(1)$ and $\p_\xi \pi^-(\zeta) = \OO_{\C^2}(1)$.

2. We use the formula \cite[(23)]{FC} to compute the Berry curvature of the low-energy eigenbundle associated to $\MM$: 
\begin{equations}
\BB(\zeta) = 
\dfrac{1}{2i\big(\xi^2 + |\var(t)|^2\big)^{3/2}} \matrice{\Re \var(t) \\ \Im \var(t) \\ \xi} \cdot \left(  \p_\xi \matrice{\Re \var(t) \\ \Im \var(t) \\ \xi} \wedge \p_t \matrice{\Re \var(t) \\ \Im \var(t) \\ \xi} \right) 
\\
=
\dfrac{1}{2i\big(\xi^2 + |\var(t)|^2\big)^{3/2}} \matrice{\Re \var(t) \\ \Im \var(t) \\ \xi} \cdot \left( \matrice{0 \\ 0 \\ 1} \wedge \matrice{\Re \var'(t) \\ \Im \var'(t) \\ 0}   \right) 
\\
=  \dfrac{1}{2i\big(\xi^2 + |\var(t)|^2\big)^{3/2}} \matrice{\Re \var(t) \\ \Im \var(t) \\ \xi} \cdot \matrice{-\Im \var'(t) \\ \Re \var'(t) \\ 0} = \dfrac{1}{2i\big(\xi^2 + |\var(t)|^2\big)^{3/2}} \cdot \Im\big(\var(t)\overline{\var'(t)} \big).
\end{equations}
If we write $\var(t) = r(t)e^{i\vp(t)}$, this reduces to
\begin{equation}
\BB(\zeta) = \dfrac{i}{2\big(\xi^2 + r(t)^2\big)^{3/2}} \cdot r(t)^2 \vp'(t).
\end{equation}
Integrating over $\xi$ and $t$, we get 
\begin{equation}\label{eq:1g}
\dfrac{1}{2\pi i} \int_{\R \times \Ss^1} \BB(\zeta) d\zeta = \dfrac{1}{4\pi} \int_{\R}\dfrac{d\xi}{\big(\xi^2+r(t)^2\big)^{3/2}}  \cdot r(t)^2 \vp'(t) dt = \dfrac{1}{2\pi} \int_0^{2\pi} \vp'(t) dt.
\end{equation}
This completes the proof because the RHS of \eqref{eq:1g} is the winding number of $\var$.
\end{proof}

\end{document}